\documentclass[12pt]{amsart}

\usepackage{amsfonts, amsthm, amsmath}

\usepackage{rotating}

\usepackage{tikz}

\usepackage{graphics}

\usepackage{amssymb}

\usepackage{amscd}

\usepackage[latin2]{inputenc}

\usepackage{t1enc}

\usepackage[mathscr]{eucal}

\usepackage{indentfirst}

\usepackage{graphicx}

\usepackage{graphics}

\usepackage{pict2e}

\usepackage{mathrsfs}

\usepackage{enumerate}
\usepackage[pagebackref]{hyperref}
\usepackage{cite}
\usepackage{color}
\usepackage{epic}
\usepackage{hyperref} 
\usepackage{framed}
\usepackage{mathabx}
\usepackage{booktabs}
\usepackage{makecell}
\newcolumntype{V}{!{\vrule width 2pt}}

\numberwithin{equation}{section}
\def\blue{\textcolor{blue}}

\def\magenta{\textcolor{magenta}}
\theoremstyle{plain}

\newtheorem{theorem}{Theorem}[section]
\newtheorem{corollary}[theorem]{Corollary}
\newtheorem{proposition}[theorem]{Proposition}
\newtheorem{conjecture}[theorem]{Conjecture}
\newtheorem{remark}[theorem]{Remark}
\newtheorem{lemma}[theorem]{Lemma}
\newtheorem{definition}[theorem]{Definition}
\newtheorem{example}[theorem]{Example}

\def\T{\mathcal{T}}
\def\B{\mathcal{B}}
\def\s{{\bf s}}
\newcommand{\leaf}{\mathsf{leaf}}
\def\deg{\mathsf{deg}}
\def\od{\mathsf{od}}
\def\ed{\mathsf{ed}}
\def\el{\mathsf{el}}
\def\o{\mathsf{odd}}
\def\fo{\mathsf{oddf}}
\def\e{\mathsf{even}}
\def\oe{\mathsf{oe}}
\def\oo{\mathsf{oo}}
\def\ee{\mathsf{ee}}
\def\eo{\mathsf{eo}}
\def\act{\mathsf{act}}
\def\eact{\mathsf{eact}}
\def\oact{\mathsf{oact}}

\def\rdeg{\mathsf{rdeg}}
\def\rol{\mathsf{rol}}
\def\ell{\mathsf{ell}}
\def\ord{\mathsf{ord}}
\def\oler{\mathsf{oler}}
\def\ndoler{\mathsf{ndoler}}
\def\eler{\mathsf{eler}}
\def\ndord{\mathsf{ndord}}

\def\dme{\mathsf{dme}}
\def\dmo{\mathsf{dmo}}

\def\Z{ \mathbb{Z}}
\def\sn{\mathrm{sn}}
\def\cn{\mathrm{cn}}
\def\dn{\mathrm{dn}}
\def\mod{\mathrm{mod}}
\newcommand{\der}{{\rm{d}}}

\begin{document}

\title[A symmetry on weakly increasing trees]{A symmetry on weakly increasing trees and \\multiset Schett polynomials}

\author[Z. Lin]{Zhicong Lin}
\address[Zhicong Lin]{Research Center for Mathematics and Interdisciplinary Sciences, Shandong University, Qingdao 266237, P.R. China}
\email{linz@sdu.edu.cn}

\author[J. Ma]{Jun Ma}
\address[Jun Ma]{ School of Mathematical Sciences, Shanghai Jiao Tong University, Shanghai 200240, P.R. China}
\email{majun904@sjtu.edu.cn}

\date{\today}

\begin{abstract}
By considering the parity of the degrees and levels of nodes in increasing trees, a new combinatorial interpretation for the coefficients of the Taylor expansions of the Jacobi elliptic functions is found. As one application of this new interpretation, a conjecture of Ma--Mansour--Wang--Yeh is solved. Unifying the concepts of increasing trees and plane trees, Lin--Ma--Ma--Zhou introduced weakly increasing trees on a multiset.  A symmetry joint distribution of ``even-degree nodes on odd levels'' and ``odd-degree nodes''  on weakly increasing trees is found, extending the Schett polynomials, a generalization  of the Jacobi elliptic functions introduced by Schett,  to multisets. A combinatorial proof and an algebraic proof of this symmetry are provided, as well as several relevant interesting consequences.  Moreover, via introducing a group action on trees, we prove the partial $\gamma$-positivity of the multiset Schett polynomials, a result implies both the symmetry and the unimodality of these polynomials.

\end{abstract}


\keywords{Jacobi elliptic functions; Weakly increasing trees; Parity;  Degrees or levels of nodes}

\maketitle


\section{Introduction}\label{sec1: intro}
As unified generalization of increasing trees and plane trees, the weakly increasing trees on a multiset were introduced recently in the joint work of the two authors with Ma and Zhou~\cite{Lin2020}. The  objective of this article is to prove both bijectively and algebraically the symmetry of the joint distribution of ``even-degree nodes on odd levels'' and ``odd-degree nodes'' on weakly increasing trees.  The discovery of this symmetric distribution was inspired by Dumont's combinatorial interpretation of the Jacobi elliptic functions~\cite{Dumont1979,Dumont1981}, Deutsch's combinatorial  bijection on plane trees~\cite{Deutsch2000} and Liu's recursive involution on increasing trees~\cite{Liu2019}. 

Let us begin with the definition of weakly increasing trees on a multiset. A {\em plane tree} $P$ can be defined  recursively as follows: a node $v$ is one designated vertex, which is called the root of $P$.
Then either $P$ contains only one node $v$, or it has a sequence $(P_1,P_2,\ldots,P_k)$ of $k$ subtrees $P_i$, each of which is a plane tree. So, the subtrees of each node are linearly ordered. We write these subtrees  in the order left to right when we draw such trees. We also write the root $v$  on the top and draw an edge from $v$ to the root of each of its subtrees. Let $(p_1,p_2,\ldots,p_n)$ be a sequence of positive integers. Denote by $M=M(p_1,p_2,\ldots,p_n)$ the multiset $\{1^{p_1},2^{p_2},\ldots,n^{p_n}\}$ and let $p=\sum_{i=1}^np_i$.

\begin{definition}[Weakly increasing trees~\cite{Lin2020}]
A weakly increasing tree on $M$ is a plane tree  such that 
\begin{enumerate}[(i)]
\item  it contains $p+1$ nodes that are labeled by elements from the multiset $M\cup\{0\}$,
\item  each sequence of labels along a path from the root to any leaf is weakly increasing, and
\item  for each node $j$ in the tree, labels of the roots of subtrees of $j$ are  weakly increasing from left to right. \end{enumerate}   Denote by $\mathcal{T}_{M}$ the set of weakly increasing trees on $M$. See Fig.~\ref{18wit} for all the  weakly increasing trees on $M=\{1^2,2^2\}$. 
\end{definition}

\begin{figure}
	\begin{center}
	\includegraphics[width=8cm,height=6cm]{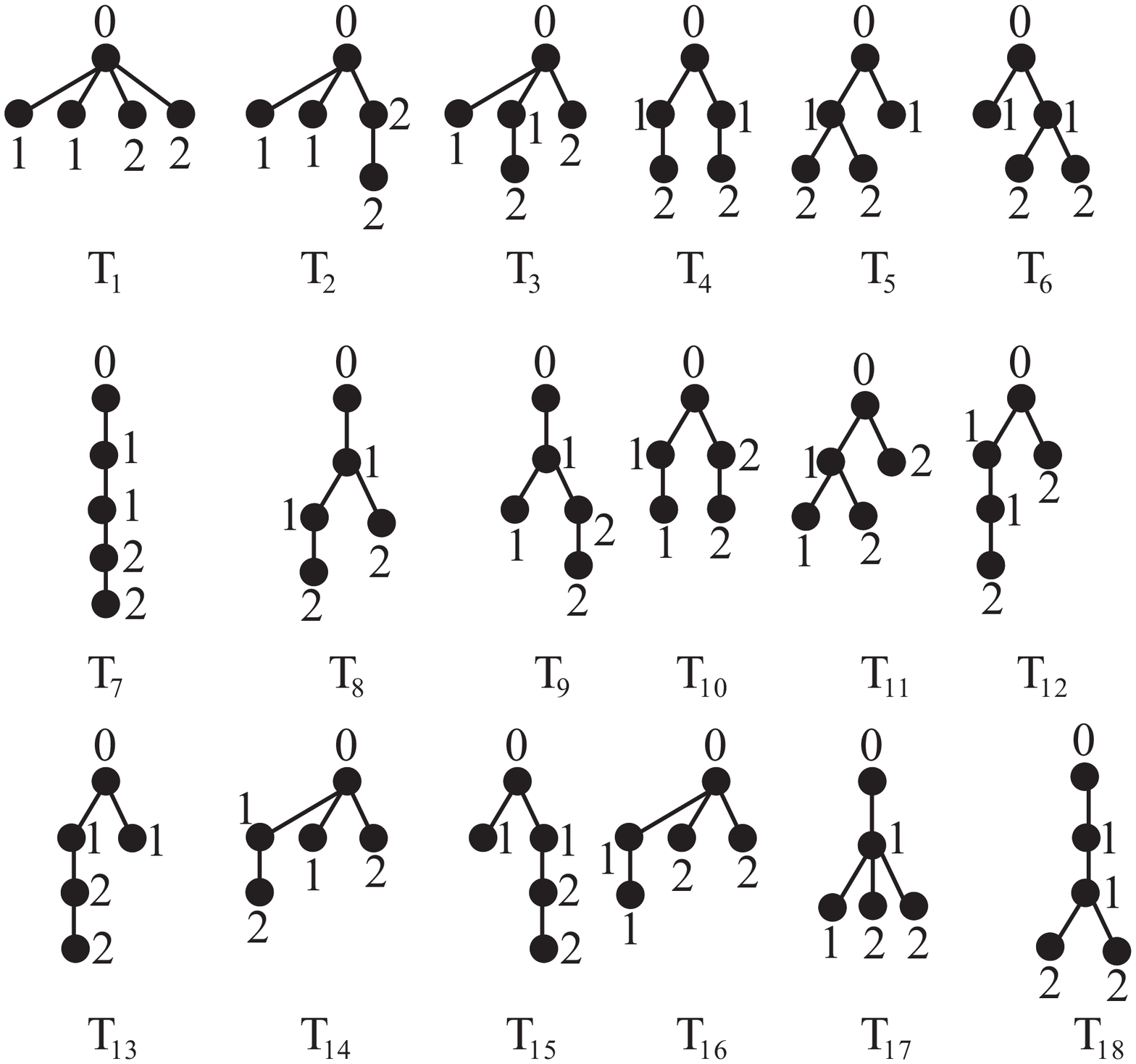}
	\end{center}
	\caption{\label{18wit}All $18$ weakly increasing trees on the multiset $\{1^2,2^2\}$.}
	\end{figure}

Note that weakly increasing trees on $[n]:=\{1,2,\ldots,n\}$ are exactly {\em increasing trees} on $[n]$, while weakly increasing trees on $\{1^n\}$ are in obvious bijection with plane trees with $n$ edges. This enables us to study the unity between plane trees and increasing trees in the framework of weakly increasing trees. The enumerative  results obtained in~\cite{Lin2020} reflect that weakly increasing trees is a natural and nice generalization of plane trees and increasing trees: 
\begin{itemize}
\item  The number of weakly increasing trees on $M(p_1,\ldots,p_n)$ has a  compact product formula 
$$
|\T_M|=\frac{1}{1+N_n}\prod_{i=1}^n{N_i+p_i\choose p_i},
$$
where $N_i:=p_1+\cdots+p_i$ for each $i\in[n]$. 
\item The {\em $M$-Eulerian--Narayana polynomial}, which  interpolates between  the Eulerian polynomial (when $M=[n]$) and the Narayana polynomial (when $M=\{1^n\}$), was defined by 
\begin{equation}\label{def:MEul}
A_{M}(t)=\sum_{T\in\T_{M}}t^{\leaf(T)},
\end{equation}
where $\leaf(T)$ denotes the number of leaves of $T$. The $\gamma$-positivity of $A_{M}(t)$ has an unified group action proof which possesses an unexpected application in interpreting the $\gamma$-coefficients of the {\em $k$-multiset Eulerian polynomials} (see~\cite{Lin2020} for $k=2$ and~\cite{Lin2021} for general $k$). 
\item  There are connections between $M$-Eulerian--Narayana polynomials for the multiset $M=\{1^2,2^2,\ldots,n^2\}$ (resp.~$M=\{1^2,2^2,\ldots,(n-1)^2,n\}$) and  Savage and Schuster's {\em$\s$-Eulerian polynomials} for the sequence $\s=(1, 1,3,2,5,3, \ldots, 2n-1,n,n+1)$ (resp.~$\s= (1, 1,3,2,5,3, \ldots, 2n-1,n)$).
\end{itemize}

In this paper, we continue to investigate the weakly increasing trees by considering several classical tree statistics related to the degrees and levels of nodes. The {\em degree} (also called {\em out-degree}) of a node in a tree  is the number of its children and the {\em level} of a node is measured by the number of edges lying on the unique path from the root to it. So the root lies at level $0$. For $T\in\T_M$, the six tree statistics on $T$ that we consider are: 
\begin{itemize}
\item The number of nodes of degree $q$ in $T$, denoted by $\deg_q(T)$.  In particular, $\deg_0(T)=\leaf(T)$. 
\item The number of nodes of degree $q$ in odd levels of $T$, denoted by $\od_q(T)$.
\item The number of nodes in even levels of $T$, denoted by $\el(T)$.
\item The number of odd-degree nodes in $T$, denoted by $\o(T)$.
\item The number of even-degree nodes in odd (resp.~even) levels of $T$, denoted by $\oe(T)$ (resp.~$\ee(T)$).
\end{itemize}

The above six tree statistics have been extensively studied in the literature for plane trees in~\cite{Chen1990,Dershowitz1980,Dershowitz1981,Deutsch2000,Eu2004,Riordan1975} and for increasing trees in~\cite{Bergeron1992,Chen2017,Kuznetsov1994,Liu2019}. Particularly, Deutsch \cite{Deutsch2000} introduced a bijection $\widehat{(\,)}$ from  the set $\mathcal{P}_n$ of plane trees with $n$ edges to itself such that for each $T\in\mathcal{P}_n$
 $$\deg_q(T)=\left\{\begin{array}{lll}\od_{q-1}(\widehat{T})&\text{if}&q\geq 1,\\
\el(\widehat{T})&\text{if}&q=0.\end{array}\right.$$
 Inspired by Deutsch's result, a natural problem arises:  is the above equidistribution still holds on increasing trees or even on weakly increasing trees? The answer to this question is affirmative and it turns out that Deutsch's bijection $\widehat{(\,)}$ can be extended\footnote{We learn that Shishuo Fu and his master student have also observed such extension independently.} to weakly increasing trees. 

%

\begin{figure}
\includegraphics[width=8.5cm,height=4cm]{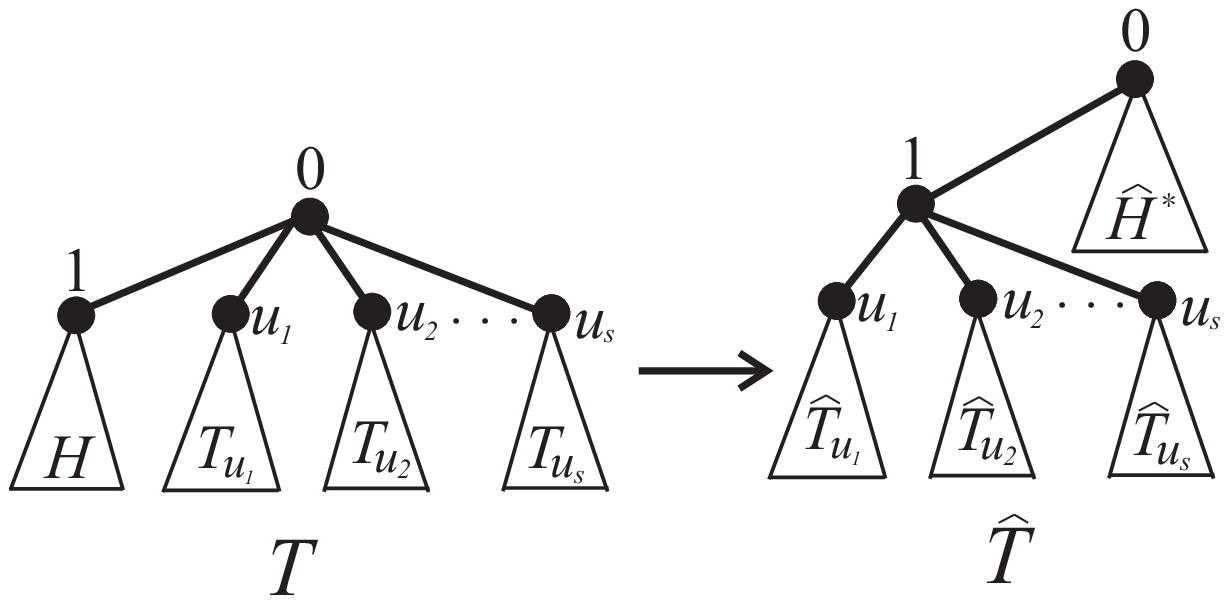}\caption{The construction of $\widehat{T}$.}\label{deg1}
\end{figure}

In fact, Deutsch's bijection $\widehat{(\,)}$ can be generalized directly to $\widehat{(\,)}=\widehat{(\,)}_M:\T_M\rightarrow\T_M$ by taking the labels of trees into account as follows. The mapping $\widehat{(\,)}$ is defined inductively.  Firstly, set $\widehat{\emptyset}=\emptyset$.
Let $T$ be a nonempty weakly increasing tree in $\mathcal{T}_M$.  For any node $v$ of $T$ let $T_v$ denote the  subtree of $T$ induced by $v$ and its descendants. Clearly, the leftmost child of the root $0$ in $T$ must be a node of label $1$ and let $H$ denote the subtree of $T$ induced by this $1$ and its descendants. 
Suppose that $u_1,\ldots,u_s$ are all children of the root $0$ other than the leftmost child in $T$ from left to right. We construct the weakly increasing tree $\widehat{T}$ as follows:
\begin{itemize}
\item Change the label $1$ of the root of $\widehat{H}$ to $0$ and denote by $\widehat{H}^*$ the resulting tree;
\item Attach an isolate node $1$ to the root $0$ of $\widehat{H}^*$ such that it is the leftmost child of $0$;
\item Make $\widehat{T}_{u_1},\ldots, \widehat{T}_{u_s}$ the branches of the above new node $1$ from left to right.
\end{itemize}
 See Fig.~\ref{deg1}   for graphical descriptions of $\widehat{T}$.
By induction on the number of edges of the tree, we obtain the following generalization of Deutsch's result.

\begin{theorem}\label{bijection-Weakly} Let $M=\{1^{p_1},2^{p_2},\ldots,n^{p_n}\}$. The mapping $\widehat{(\,)}:\T_M\rightarrow\T_M$ is a bijection satisfying 
$$\deg_q(T)=\left\{\begin{array}{lll}\od_{q-1}(\widehat{T})&\text{if}&q\geq 1\\
\el(\widehat{T})&\text{if}&q=0\end{array}\right.$$
for any $T\in\mathcal{T}_M.$
\end{theorem}
 Theorem~\ref{bijection-Weakly} is new even for increasing trees, i.e., when $M=[n]$. It reflects again the wonderful unity between plane trees and increasing trees. One interesting consequence of Theorem~\ref{bijection-Weakly} is 
 $$
 \sum_{T\in\T_{M}}t^{\leaf(T)}=\sum_{T\in\T_{M}}t^{\el(T)},
 $$
 which provides a new interpretation of the $M$-Eulerian--Narayana polynomial $A_M(t)$ defined by~\eqref{def:MEul}. This interpretation for the classical Eulerian polynomials in terms of increasing trees was known in~\cite{Br1}. 
Noticing 
$$
\o(T)=\sum\limits_{q\text{ odd}}\deg_q(T)\quad\text{and}\quad \oe(T)=\sum\limits_{q\text{  even}}\od_{q}(T),
$$
another interesting corollary of Theorem~\ref{bijection-Weakly} is the following equidistribution. 

\begin{corollary}\label{cor:WIT}
Fix a multiset $M$. For any $T\in\T_M$,  $\o(T)=\oe(\widehat{T})$. 
\end{corollary}
On the other hand, Liu~\cite{Liu2019} constructed recursively an involution on $\T_n$, the set of increasing trees on $[n]$, which proves a refined symmetry extension of Corollary~\ref{cor:WIT} for $M=[n]$. 

\begin{theorem}[Liu~\cite{Liu2019}]\label{thm:liu} There exists an involution $\phi$ on $\T_n$ such that 
$$
(\o,\oe,\ee)(T)=(\oe,\o,\ee)(\phi(T)) \quad\text{for each $T\in\T_n$}.
$$
\end{theorem}

It is Corollary~\ref{cor:WIT} and Theorem~\ref{thm:liu} that motivates the following refined symmetry distribution for weakly increasing trees. 
 \begin{theorem}\label{weak-o-oe-ee} Let $M=\{1^{p_1},2^{p_2},\ldots,n^{p_n}\}$. There exists an involution $\widetilde{(\,)}:\T_M\rightarrow\T_M$ such that 
$$
(\o,\oe,\ee)(T)=(\oe,\o,\ee)(\widetilde{T}) \quad\text{for each $T\in\T_M$}.
$$
Consequently, 
\begin{equation}\label{Weak:thm}
\sum_{T\in\mathcal{T}_M}x^{\ee(T)}y^{\oe(T)}z^{\o(T)}=\sum_{T\in\mathcal{T}_M}x^{\ee(T)}y^{\o(T)}z^{\oe(T)}.\end{equation}
 \end{theorem}
 
 Our construction of involution  $\widetilde{(\,)}$ on weakly increasing trees, having close flavor as Deutsch's bijection $\widehat{(\,)}$, is essentially different with Liu's involution $\phi$ on increasing trees. We have no idea how Liu's involution can be extended to weakly increasing trees.
 
 The {\em Jacobi elliptic function} $\sn(u,\alpha)$ may be (see~\cite{Flajolet1989}) defined by the inverse of an elliptic integral:
$$
\sn(u,\alpha)=y\quad\text{iff}\quad u=\int_0^y\frac{\der t}{\sqrt{(1-t^2)(1-\alpha^2t^2)}},
$$
where $\alpha\in(0,1)$ is a real number. The other two functions are given by 
$$
\cn(u,\alpha)=\sqrt{1-\sn^2(u,\alpha)}\quad\text{and}\quad \dn(u,\alpha)=\sqrt{1-\alpha^2\sn^2(u,\alpha)}.
$$
 The Taylor expansions of $\sn(z,\alpha), \cn(z,\alpha)$ and $\dn(z,\alpha)$ read 
\begin{align*}
\sn(u,\alpha)&=u-(1+\alpha^2)\frac{u^3}{3!}+(1+14\alpha^2+\alpha^4)\frac{u^5}{5!}\\
&\quad-(1+135\alpha^2+135\alpha^4+\alpha^6)\frac{u^7}{7!}+\cdots, \\
\cn(z,\alpha)&=1-\frac{u^2}{2!}+(1+4\alpha^2)\frac{u^4}{4!}-(1+44\alpha^2+16\alpha^4)\frac{u^6}{6!}+\cdots,\\
\dn(z,\alpha)&=1-\alpha^2\frac{u^2}{2!}+(4+\alpha^2)\frac{u^4}{4!}-(16+44\alpha^2+\alpha^4)\frac{u^6}{6!}+\cdots.
\end{align*} 
 
 Consider the derivative operator on the polynomials in three variables defined by 
 $$
 D(x)=yz,\quad D(y)=xz\quad\text{and}\quad D(z)=xy.
 $$
The polynomial $S_n(x,y,z):=D^n(x)$ is known as the {\em$n$-th Schett polynomial} (according to Dumont~\cite{Dumont1979}) which has the form:
\begin{align*}
S_{2m}(x,y,z)&=\sum_{i,j\geq0} s_{2m,i,j}x^{2i+1}y^{2j}z^{2m-2i-2j}, \\
S_{2m+1}(x,y,z)&=\sum_{i,j\geq0} s_{2m+1,i,j}x^{2i}y^{2j+1}z^{2m+1-2i-2j}.
\end{align*} 
The first values of $S_n(x,y,z)$ are: 
\begin{align*}
S_0&=x,\quad S_1=yz,\quad S_2=xy^2+xz^2,\quad S_3=y^3z+yz^3+4x^2yz,\\
S_4&=xy^4+14xy^2z^2+xz^4+4x^3y^2+4x^3z^2.
\end{align*}
The main result of Schett~\cite{Schett1976} was to prove that 
\begin{itemize}
\item the coefficient of $(-1)^n\alpha^{2j}u^{2n+1}/(2n+1)!$ in the Taylor expansion of $\sn(u,\alpha)$ equals $s_{2n,0,j}$ or $s_{2n+1,0,j}$;
 \item the coefficient of $(-1)^n\alpha^{2i}u^{2n}/(2n)!$ (resp.~$(-1)^n\alpha^{2n-2i}u^{2n}/(2n)!$) in the Taylor expansion of $\cn(u,\alpha)$ (resp.~$\dn(u,\alpha)$) equals $s_{2n-1,i,0}$ or $s_{2n,i,0}$.
 \end{itemize}
 This indicates that Schett polynomials is a generalization of the Jacobi elliptic functions.
 Schett also noticed that $S_m(1,1,1)=\sum_{i,j\geq0} s_{m,i,j}=m!$, which inspired Dumont~\cite{Dumont1979,Dumont1981} to find the first interpretation of the Schett polynomials in terms of the parity of {\em cycle peaks} in permutations. Our consideration of the parity of the degrees and levels of nodes in increasing trees leads to a new interpretation of the Schett polynomials $S_n(x,y,z)$. 
 \begin{theorem}\label{int:schett}
 The coefficient $s_{m,i,j}$ of the  Schett polynomials $S_m(x,y,z)$ counts the number of trees $T\in\T_m$ such that $\lfloor\ee(T)/2\rfloor=i$ and $\lfloor\oe(T)/2\rfloor=j$, namely, 
 $$
 S_m(x,y,z)=\sum_{T\in\mathcal{T}_m}x^{\ee(T)}y^{\oe(T)}z^{\o(T)}. 
 $$
 \end{theorem}
 In view of Theorem~\ref{int:schett}, the joint distribution in~\eqref{Weak:thm}
 $$
 S_M(x,y,z):=\sum_{T\in\mathcal{T}_M}x^{\ee(T)}y^{\oe(T)}z^{\o(T)}
 $$ 
  is named the {\em multiset Schett polynomial}  for the multiset $M$, extending the Jacobi elliptic functions from sets to multisets. In particular, we connect the Jacobi elliptic functions for plane trees to fighting fish with a marked tail originally  studied by  Duchi,  Guerrini,  Rinaldi and  Schaeffer~\cite{Duchi2017}. 
 
 Gamma-positive polynomials arise frequently in enumerative combinatorics and have recent impetus coming from enumerative geometry; see the two recent surveys by Br\"and\'en~\cite{Br2} and by Athanasiadis~\cite{Ath}. A univariate polynomial $f(x)$ of degree $n$  is said to be {\em$\gamma$-positive} if it can be expanded as 
 $$
 f(x)=\sum_{k=0}^{\lfloor\frac{n}{2}\rfloor}\gamma_{k}x^k(1+x)^{n-k}
 $$
 with $\gamma_k\geq0$. If such an expansion exists, then $f(x)$ is also palindromic (or symmetric) and unimodal. 
 Moreover, the $\gamma$-coefficients $\gamma_k$ usually (but not always) have nice combinatorial interpretations, which makes this theme  even more charming. The classical Eulerian polynomials, as well as their various generalizations~\cite{Br1,Lin2015,Lin2020}, are some of the typical examples arising from permutation statistics. 
 
 A bivariate polynomial $h(x,y)$ is said to be {\em homogeneous $\gamma$-positive}, if $h(x,y)$ is homogeneous and $h(x,1)$ is $\gamma$-positive. A trivariate polynomial $p(x,y,z)=\sum_{i}x^i s_i(y,z)$ is {\em partial $\gamma$-positive} if every $s_i(y,z)$ is homogeneous $\gamma$-positive. There has been recent interest in investigating partial $\gamma$-positive polynomials with combinatorial meanings~\cite{MMY2019,Ma2021,LMZ2021}. Generalizing the symmetry in~\eqref{Weak:thm}, we aim to prove the  partial $\gamma$-positivity of the {\em reduced multiset Schett polynomial}
 \begin{equation}\label{def:redsche}
 \hat{S}_M(x,y,z):=\sum_{T\in\mathcal{T}_M}x^{\lfloor\frac{\ee(T)}{2}\rfloor}y^{\lfloor\frac{\oe(T)}{2}\rfloor}z^{\lfloor\frac{\o(T)}{2}\rfloor}. 
 \end{equation}
 Noticing the relationships 
 \begin{equation}\label{rel:reduce}
 \#\{\text{nodes in $T$}\}\equiv \ee(T) \,\,(\mod\, 2)\quad\text{and}\quad \#\{\text{nodes in $T$}\}=\ee(T)+\oe(T)+\o(T),
 \end{equation}
  we see that the  reduced multiset Schett polynomial  $\hat{S}_M(x,y,z)$ encodes essentially the same information as $S_M(x,y,z)$. 
  
  \begin{theorem}\label{thm:action}
 Let $M=\{1^{p_1},2^{p_2},\ldots,n^{p_n}\}$ be a multiset with $p=\sum_{i=1}^np_i$.  The reduced multiset Schett polynomial $\hat{S}_M(x,y,z)$ has the partial $\gamma$-positivity expansion
\begin{equation}\label{eq:partial-gm}
\hat{S}_M(x,y,z)=\sum_{T\in\mathcal{T}_M}x^{\lfloor\frac{\ee(T)}{2}\rfloor}y^{\lfloor\frac{\oe(T)}{2}\rfloor}z^{\lfloor\frac{\o(T)}{2}\rfloor}=\sum_ix^i\sum_j\hat\gamma_{M,i,j}(yz)^j(y+z)^{\lfloor\frac{p}{2}\rfloor-i-2j},
\end{equation}
where $\hat\gamma_{M,i,j}$ enumerates  weakly increasing trees $T$ on $M$ with $\lfloor\frac{\ee(T)}{2}\rfloor=i$, $\act(T)=\lfloor\frac{p}{2}\rfloor-i-2j$ and $\eact(T)=0$. Consequently, 
\begin{itemize}
\item the polynomial $\hat{S}_M(x,y,z)$ is symmetric in $y$ and $z$, which is equivalent to the symmetry  in~\eqref{Weak:thm};
\item if $\hat{S}_M(x,y,z)=\sum_ix^i\hat{S}_{M,i}(y,z)$, then $\hat{S}_{M,i}(y,1)=\hat{S}_{M,i}(1,y)$ is palindromic and unimodal. 
\end{itemize}
 \end{theorem}
 
 The two statistics $\act$ and $\eact$ concerned are defined as follows.

\begin{definition}\label{DEF:wit-active}
 Let $T\in\mathcal{T}_{M}$ and let $u$ be a node of $T$.
 Suppose that  $y$ (possibly empty) is the first brother to the right of $u$.  The node $u$ is {\bf\em active} if it satisfies the following three conditions:
\begin{enumerate}[(a)]
\item  the  level  of $u$ is odd;
\item $u$ is the $k$-th child (from left to right) of its parent with $k$ being odd;
\item either (i) $y$ is not empty and the degrees of $u$ and $y$ have the same parity, or (ii) $y$ is empty  and the degree of $u$  is odd.
\end{enumerate}
 Furthermore, $u$ is said to be {\bf\em active odd} or {\bf\em active even} according to the parity of the degree of $u$. Let $\act(T)$ (resp.~$\eact(T)$, $\oact(T)$) be the number of active (resp.~active even, active odd) nodes of $T$. 
\end{definition}

\begin{example}
Take $M=\{1^2,2^2\}$ for an example of expansion~\eqref{eq:partial-gm}, there are $18$ trees as displayed in Fig.~\ref{18wit}, which gives 
$$
\hat{S}_M(x,y,z)=3x(y+z)+(y+z)^2+8yz.
$$
The $\gamma$-coefficients $3$ counts the trees $T_4,T_{10}$ and $T_{17}$; $1$ counts $T_{7}$; $8$ counts the trees $T_3,T_8,T_9$ and the trees from $T_{12}$ to $T_{16}$. 
\end{example}

 The purpose of this paper is  threefold. One is to construct the involution $\widetilde{(\,)}, :\T_M\rightarrow\T_M$ to prove combinatorially  the refined symmetry~\eqref{Weak:thm}, which is provided in Section~\ref{com:sym}, as well as several relevant interesting consequences. The second is to study the algebraic aspect of the refined symmetry~\eqref{Weak:thm}, including proofs of Theorem~\ref{int:schett} and a conjecture of  Ma--Mansour--Wang--Yeh~\cite{Ma2019} via Chen's context-free grammar and a generating function proof of~\eqref{Weak:thm}, which are fulfilled in 
 Section~\ref{alg:sym}. The last is to develop a group action on weakly increasing binary trees to prove Theorem~\ref{thm:action}, which forms the content of Section~\ref{sec:4}.
 This paper reflects appropriately {\it again} (cf.~\cite{Foata})  the deep ideas of our  great master, M.P. Sch\"utzenberger: every algebraic relation is to be given a combinatorial counterpart and vice versa.

 \section{Combinatorics of the symmetry~\eqref{Weak:thm}}
 \label{com:sym}
 In this section, we construct the involution  $\widetilde{(\,)}:\T_M\rightarrow\T_M$ for Theorem~\ref{weak-o-oe-ee}. 
It will be defined recursively. 

For $M=\emptyset\text{ and }\{1\}$,  we let $\widetilde{T}=T$
 for any $T\in \mathcal{T}_M$.  If $M=[2]$ and $\{1^2\}$, we let  \begin{center}\begin{tabular}{|c|c||c|c|}\hline $T$&$\widetilde{T}$&$T$&$\widetilde{T}$\\
 \hline \includegraphics[width=0.3cm,height=0.8cm]{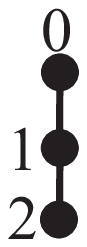}&\includegraphics[width=0.65cm,height=0.8cm]{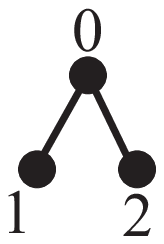}&\includegraphics[width=0.8cm,height=0.8cm]{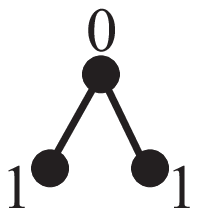}&\includegraphics[width=0.3cm,height=0.8cm]{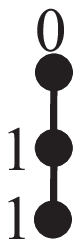}\\
\hline  \includegraphics[width=0.65cm,height=0.8cm]{2-it-1.eps}&\includegraphics[width=0.3cm,height=0.8cm]{2-it-2.eps}&\includegraphics[width=0.3cm,height=0.8cm]{2-wit-2.eps}&\includegraphics[width=0.8cm,height=0.8cm]{2-wit-1.eps}\\
\hline \end{tabular}.\end{center}
 For a general fixed multiset $M$ and $T\in\T_M$, we have the decomposition of $T$ (see the left graph of Fig.~\ref{invo}) as: 
 \begin{itemize}
 \item The leftmost child of the root $0$ is a node with label $1$. Let $x$ and $y$  be the leftmost child and the closest sibling of this special $1$ in $T$. It is possible that $x$ or $y$ may not exist, which does not affect our construction. 
 \item Let $v_1,\ldots,v_s$ be all the siblings of $x$ from left to right. Denote $F=T_x$ the subtree of $T$ induced by $x$ and its descendants.
 \item Let $u_1,\ldots,u_t$ be all the children  of $y$ from left to right. Denote $H$ the subtree of $T$ induced by the root and its children to the right of $y$ together with their descendants.
 \end{itemize}
 \begin{figure}[h!]
\includegraphics[width=14cm,height=4.8cm]{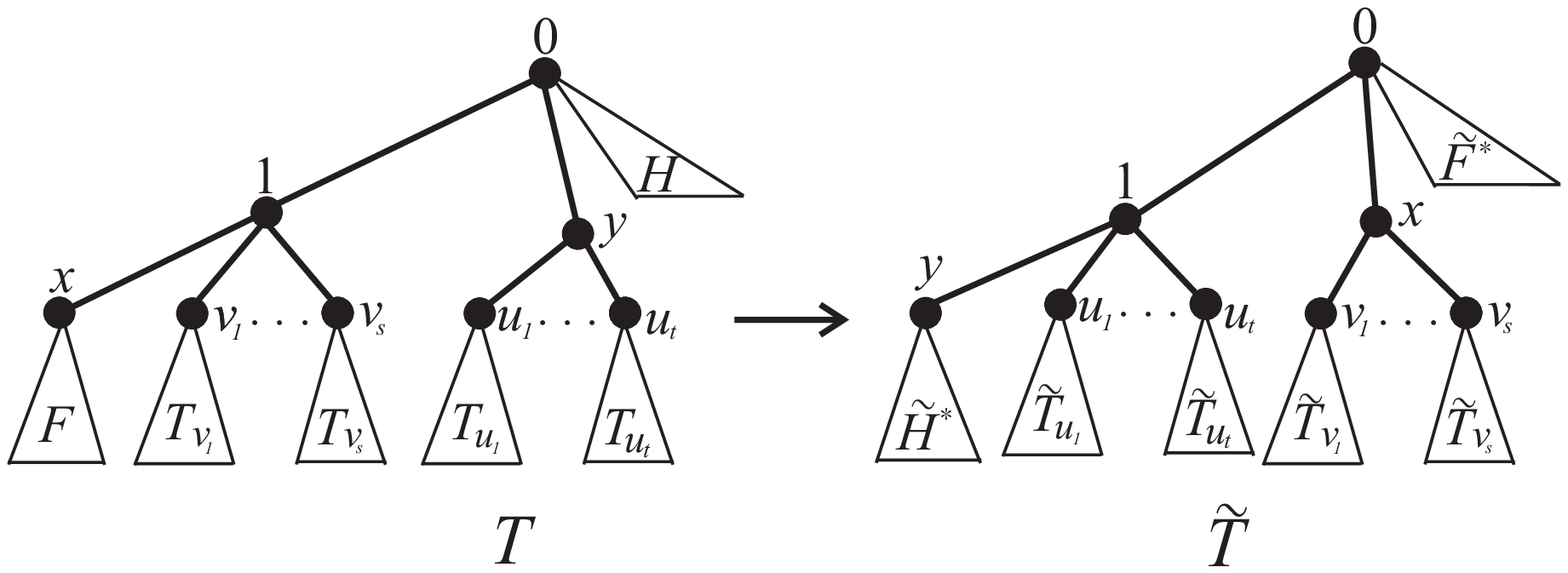}\caption{The construction of $\widetilde{T}$.}\label{invo}
\end{figure}
 The main idea underling the construction is to exchange the role of $x$ and $y$ and further exchange their siblings with their children. Let us define $\widetilde{T}$ recursively by the following steps (see again Fig.~\ref{invo} for nice visualization of our involution):
 \begin{itemize}
 \item[i)] Change the label $x$ of the root of $\widehat{F}$ to $0$ and denote by $\widehat{F}^*$ the resulting tree;
 \item[ii)] Attach a node $1$ and a node $x$ at the root  $0$ of $\widehat{F}^*$ as its first and second children, respectively; 
 \item[iii)] Change the label $0$ of the root of $\widehat{H}$ to $y$ and attach the resulting tree, denoted $\widehat{H}^*$, at the node $1$ as its leftmost branch; further attach the subtrees $\widetilde{T}_{u_1},\ldots,\widetilde{T}_{u_t}$ at the node $1$ as its branches from left to right;
 \item[iv)] Finally, attach the subtrees $\widetilde{T}_{v_1},\ldots,\widetilde{T}_{v_s}$ at the node $x$ as its branches from left to right and let the resulting tree be $\widetilde{T}$ (see the right graph in Fig.~\ref{invo}). 
 \end{itemize}
 For example, for $M=\{1^2,2^2\}$ in Fig.~\ref{18wit}, the mapping $\widetilde{(\,)}$ works as $\widetilde{T}_1=T_7, \widetilde{T}_2=T_{18}, \widetilde{T}_3=T_8, \widetilde{T}_4=T_{11}, \widetilde{T}_5=T_{10}, \widetilde{T}_6=T_{17}, \widetilde{T}_9=T_{15}, \widetilde{T}_{12}=T_{14}$ and $\widetilde{T}_{13}=T_{16}$.
  It is clearly from the above construction that $\widetilde{T}$ is a weakly increasing trees in $\T_M$ and so the mapping $\widetilde{(\,)}:\T_M\rightarrow\T_M$ is well-defined. Moreover, this mapping is an involution, as evident from the construction by induction on the size of $M$. 
 
 It remains to verify that 
 \begin{equation}\label{stat:invo}
(\o,\oe,\ee)(T)=(\oe,\o,\ee)(\widetilde{T}).
\end{equation}
 Let us first consider the number of odd-degree nodes. We need to distinguish three cases: 
 \begin{enumerate}
\item If both $x$ and $y$ exist, then
 \begin{align*}
 \o(T)&=\o(H)+\o(F)+\sum_{i=0}^s\o(T_{v_i})+\sum_{i=0}^t\o(T_{u_i})+\chi(\text{$s$ even})+\chi(\text{$t$ odd})\\
 &=\oe(\widetilde{H}^*)+\oe(\widetilde{F}^*)+\sum_{i=0}^s\oe(\widetilde{T}_{v_i})+\sum_{i=0}^t\oe(\widetilde{T}_{u_i})+\chi(\text{$s$ even})+\chi(\text{$t$ odd})=\oe(\widetilde{T}),
 \end{align*} 
 where $\chi(\mathsf{S})$ equals $1$, if the statement $\mathsf{S}$ is true; and $0$, otherwise. 
 \item If $x$ does not exist, then 
  \begin{align*}
 \o(T)&=\o(H)+\sum_{i=0}^s\o(T_{v_i})+\chi(\text{$t$ odd})\\
 &=\oe(\widetilde{H}^*)+\sum_{i=0}^t\oe(\widetilde{T}_{u_i})+\chi(\text{$t$ odd})=\oe(\widetilde{T}). 
 \end{align*} 
\item If $x$ does not exist, then 
 \begin{align*}
 \o(T)&=1+\o(F)+\sum_{i=0}^s\o(T_{v_i})+\chi(\text{$s$ even})\\
 &=1+\oe(\widetilde{F}^*)+\sum_{i=0}^s\oe(\widetilde{T}_{v_i})+\chi(\text{$s$ even})=\oe(\widetilde{T}). 
 \end{align*} 
 \end{enumerate}
 In either case, we have 
 $$\o(T)=\oe(\widetilde{T}).$$
 Since $\widetilde{(\,)}$ is an involution, it follows that 
 $$
 \oe(T)=\o(\widetilde{T})
 $$
  also holds. Finally, we need to consider the number of even-degree nodes on even levels: 
  \begin{align*}
 \ee(T)&=\ee(H)+\ee(F)+\sum_{i=0}^s\ee(T_{v_i})+\sum_{i=0}^t\ee(T_{u_i})\\
 &=\ee(\widetilde{H}^*)+\ee(\widetilde{F}^*)+\sum_{i=0}^s\ee(\widetilde{T}_{v_i})+\sum_{i=0}^t\ee(\widetilde{T}_{u_i})=\ee(\widetilde{T}). 
 \end{align*} 
 This finishes the proof that the involution $\widetilde{(\,)}:\T_M\rightarrow\T_M$ has the feature~\eqref{stat:invo} and provides the combinatorial proof of symmetry~\eqref{Weak:thm}.

 \subsection{Relevant consequences}
 Our construction of $\widetilde{(\,)}$ is  more intuitive than Liu's involution on $\T_n$ and provides an alternative approach to his refined symmetry on increasing trees
  \begin{equation}\label{sym:inct}
\sum_{T\in\mathcal{T}_n}x^{\o(T)}y^{\oe(T)}z^{\ee(T)}=\sum_{T\in\mathcal{T}_n}x^{\oe(T)}y^{\o(T)}z^{\ee(T)}.\end{equation}
 On the other hand, our involution $\widetilde{(\,)}$ for $M=\{1^n\}$ provides a combinatorial proof of the following refined symmetry for plane trees, which seems new to the best of our knowledge. 
 \begin{corollary} For $n\geq1$, we have the refined symmetry
 \begin{equation}\label{sym:plane}
\sum_{T\in\mathcal{P}_n}x^{\o(T)}y^{\oe(T)}z^{\ee(T)}=\sum_{T\in\mathcal{P}_n}x^{\oe(T)}y^{\o(T)}z^{\ee(T)}.\end{equation}
 \end{corollary}
 
 For a tree $T\in\T_M$, let $\e(T)$ and $\oo(T)$ be the number of {\em even-degree nodes} and the number of {\em odd-degree nodes  in  odd levels} in $T$, respectively. We are interested in the symmetry of the parameter $\e(T)$ compared with $\o(T)$ in Theorem~\ref{weak-o-oe-ee}.
Note that 
$$
\e(T)=\sum\limits_{q\text{ even}}\deg_q(T)\quad\text{and}\quad \oo(T)=\sum\limits_{q\text{ odd}}\od_{q}(T).
$$
 It then follows from Theorem~\ref{bijection-Weakly} that 
 \begin{equation}\label{Para2}
 \sum_{T\in\mathcal{T}_M}x^{\e(T)}=\sum_{T\in\mathcal{T}_M}x^{\oo(T)+\el(T)}.
 \end{equation}
In fact, replacing  $z$ with $xyz$ in identity~\eqref{Weak:thm} and  noticing  $\o(T)+\ee(T)=\oo(T)+\el(T)$ and $\oe(T)+\ee(T)=\e(T)$, we immediately obtain the following interesting corollary.

 \begin{corollary}\label{weak-e-oo} Let $M=\{1^{p_1},2^{p_2},\ldots,n^{p_n}\}$. Then
 $$
 \sum_{T\in\mathcal{T}_M}x^{\e(T)}y^{\oo(T)+\el(T)}z^{\ee(T)}=\sum\limits_{T\in\mathcal{T}_M}x^{\oo(T)+\el(T)}y^{\e(T)}z^{\ee(T)}.
 $$
 \end{corollary}

For a tree $T\in\T_M$, we are also interested in some variations of the three statistics `$\o$', `$`\oe$' and `$\ee$' where the root of $T$ is not taken into account.  
 \begin{itemize}
 \item  The number of odd-degree nodes other than the root in $T$, denoted $\o^*(T)$.
\item The number of even-degree nodes other than the root in odd (resp.~even) level of $T$, denoted by $\oe^*(T)$ (resp.~$\ee^*(T)$). Note that $\oe^*(T)=\oe(T)$. 
 \end{itemize}
 Surprisingly, we still have the following refined symmetry for `$\o$', `$`\oe$' and `$\ee$', as a variation of~\eqref{Weak:thm}. 
 
\begin{corollary}\label{weak-nonzero} Let $M=\{1^{p_1},2^{p_2},\ldots,n^{p_n}\}$. Then
 $$\sum_{T\in\mathcal{T}_M}x^{\o^*(T)}y^{\oe^*(T)}z^{\ee^*(T)}=\sum_{T\in\mathcal{T}_M}x^{\ee^*(T)}y^{\oe^*(T)}z^{\o^*(T)}.$$
\end{corollary}
 
\begin{figure}
\includegraphics[width=8.8cm,height=3cm]{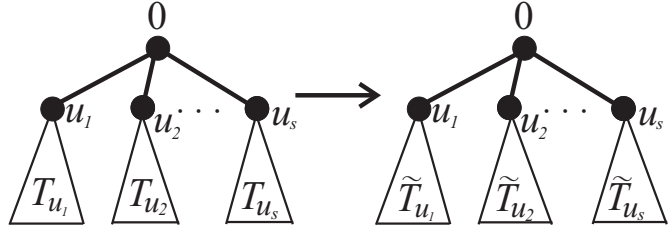}\caption{The construction of $\psi(T)$.}\label{C6}
\end{figure}
 
 \begin{proof}
 For any weakly increasing tree $T\in\mathcal{T}_M$, suppose that $u_1,\ldots,u_s$ are all children of the root $0$ of $T$.  We  construct a weakly increasing tree, denoted  $\psi(T)$, by 
   attaching $\widetilde{T}_{u_1},\ldots, \widetilde{T}_{u_s}$ as  subtrees of a new root $0$ from left to right.  
 See Fig.~\ref{C6} for the graphical description of $\psi(T)$.
By  Theorem~\ref{weak-o-oe-ee}, the mapping $\psi: \T_M\rightarrow\T_M$ is an involution satisfying  
$$
\o^*(T)=\ee^*(\psi(T)),\quad \oe^*(T)=\oe^*(\psi(T))\quad\text{and}\quad \ee^*(T)=\o^*(\psi(T)),
$$ 
which completes the proof.
 \end{proof}
 
 \begin{remark}
 Corollaries~\ref{weak-e-oo} and~\ref{weak-nonzero} are new even for plane trees or increasing trees. 
 \end{remark}
 
 The {\em Euler numbers} $\{E_n\}_{n\geq0}$  can be defined as  the coefficients of the Taylor expansion 
 $$
 \sec(x)+\tan(x)=\sum_{n\geq0} E_n\frac{x^n}{n!}.
 $$
 It was Andr\'e~\cite{Andre} in $1879$ that first discovered the interpretation of $E_n$ as the number of {\em alternating (down-up) permutations} of length $n$. Since then, the combinatorics of the Euler numbers has been investigated extensively; see the work~\cite{Foata-Han} of Foata and Han and the references therein.  Kuznetsov, Pak and Postnikov~\cite[Theorem~3]{Kuznetsov1994} showed that increasing trees $T\in\T_n$ with $\o^*(T)=0$ is enumerated by $E_n$. Combining this with   the increasing tree case of Corollary~\ref{weak-nonzero} results in  the following new interpretation of $E_n$. 
 
 \begin{corollary}
 The number of  increasing trees $T\in\T_n$ with $\ee^*(T)=0$ is the $n$th Euler number $E_n$. 
 \end{corollary}
 
  For $T\in\T_M$ and $v$ a node of $T$, the {\em full-degree} of $v$ is the number of nodes in $T$ to which $v$ is adjacent.  Thus, the full-degree of $v$ equals the degree of $v$ plus one, unless $v$ is the root. Deutsch and Shapiro~\cite[p.~259]{Deutsch2001} asked for a direct two-to-one
correspondence for proving combinatorially the following known property. 

\begin{proposition}[\text{See~\cite[p.~259]{Deutsch2001}}]\label{pro:twice}
Over all plane trees with n edges, the total number of nodes with odd full-degree is twice the total number of nodes with odd degree. 
\end{proposition}
 
 Response to the problem raised by Deutsch and Shapiro, Eu, Liu and Yeh~\cite{Eu2004} constructed such a  two-to-one correspondence. As an application of the two involutions $\widetilde{(\,)}$ and $\psi$, we have been able to obtain a new two-to-one correspondence proof of Proposition~\ref{pro:twice} which admits extension to weakly increasing trees perfectly.  
 
 \begin{proposition}\label{wit:twice}
Fix a multiset $M$. Over all weakly increasing trees in $\T_M$, the total number of nodes with odd full-degree is twice the total number of nodes with odd degree. 
\end{proposition}

\begin{proof}
Let $\fo(T)$ denote the number of nodes with odd full-degree in $T$. Notice that 
$$
\fo(T)=\blue{\oe(T)}+\magenta{\ee^*(T)+\chi(\text{Root degree of $T$ is odd})}.
$$
On the one hand, the involution  $\widetilde{(\,)}:\T_M\rightarrow\T_M$ satisfies 
$$
\o(T)=\blue{\oe(\widetilde{T})}. 
$$
On the other hand, the involution  $\psi:\T_M\rightarrow\T_M$ satisfies 
\begin{align*}
\o(T)&=\o^*(T)+\chi(\text{Root degree of $T$ is odd})\\
&=\magenta{\ee^*(\psi(T))+\chi(\text{Root degree of $\psi(T)$ is odd})}. 
\end{align*}
Combining the above observations, the two involutions $\widetilde{(\,)}$ and $\psi$ severe as a two-to-one correspondence proof of the desired property for all weakly increasing trees in $\T_M$. 
\end{proof}
\begin{remark}
We could not extend Eu, Liu and Yeh's two-to-one correspondence~\cite{Eu2004} from plane trees to weakly increasing trees to prove Proposition~\ref{wit:twice}. To the best of our knowledge, Proposition~\ref{wit:twice} is new even for increasing trees. 
\end{remark} 

We can refine Proposition~\ref{wit:twice} by constructing another  two-to-one correspondence in the same spirit as that in Proposition~\ref{wit:twice}, but with the bijection $\widehat{(\,)}$ replacing the involution $\widetilde{(\,)}$. 

\begin{theorem}\label{wit:twice2}
Fix a multiset $M$ and a nonnegative integer $d$. Over all weakly increasing trees in $\T_M$, the total number of nodes with  full-degree $2d+1$ is twice the total number of nodes with  degree $2d+1$. 
\end{theorem}
\begin{figure}[h!]
\includegraphics[width=8.8cm,height=3cm]{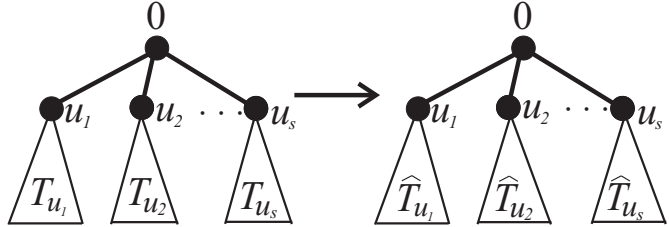}\caption{The construction of $\theta(T)$.}\label{2d1}
\end{figure}
\begin{proof}
Recall that $\od_q(T)$ is the number of nodes of degree $q$ in odd levels of a tree $T$. For our purpose, let  $\ed^*_q(T)$ be the number of nodes, other than the root, of degree $q$ in even levels of $T$. Then, the number of nodes in $T$ with full-degree $2d+1$ is
$$
\blue{\od_{2d}(T)}+\magenta{\ed^*_{2d}(T)+\chi(\text{Root degree of $T$ is $2d+1$})}. 
$$
On the one hand, the bijection $\widehat{(\,)}:\T_M\rightarrow\T_M$ defined in the introduction satisfies 
$$
\deg_{2d+1}(T)=\blue{\od_{2d}(\widehat{T})}. 
$$
On the other hand, we aim to define a bijection  $\theta:\T_M\rightarrow\T_M$ satisfies 
$$
\deg_{2d+1}(T)=\magenta{\ed^*_{2d}(\theta(T))+\chi(\text{Root degree of $\theta(T)$ is $2d+1$})},
$$
which together with  $\widehat{(\,)}$ would  severe as a two-to-one correspondence proof of the desired result. 

It remains to construct the required  bijection $\theta$. For a tree $T\in\mathcal{T}_M$, suppose that $u_1,\ldots,u_s$ are all children of the root $0$ of $T$.  We  construct the weakly increasing tree $\theta(T)$, by 
   attaching $\widehat{T}_{u_1},\ldots, \widehat{T}_{u_s}$ as  subtrees of a new root $0$ from left to right.  
 See Fig.~\ref{2d1} for the graphical description of $\theta(T)$.
By  Theorem~\ref{bijection-Weakly}, the mapping $\theta: \T_M\rightarrow\T_M$ is a bijection satisfying  
\begin{align*}
\deg_{2d+1}(T)&=\chi(s=2d+1)+\sum_{i=1}^s\deg_{2d+1}(T_{u_i})\\
&=\chi(s=2d+1)+\sum_{i=1}^s\od_{2d}(\widehat{T}_{u_i})\\
&=\magenta{\chi(\text{Root degree of $\theta(T)$ is $2d+1$})+\ed^*_{2d}(\theta(T))},
\end{align*} 
as desired. 
\end{proof}

 \section{Algebraic aspect of the symmetry~\eqref{Weak:thm}}
\label{alg:sym}
 In this section, we study the algebraic aspect of the symmetry~\eqref{Weak:thm}. We will show Theorem~\ref{int:schett} by employing Chen's context-free grammar. As an application, a conjecture posed at the end of~\cite{Ma2019} is solved.    In addition, we manage to extend a regular generating function proof of~\eqref{Weak:thm} for plane trees to weakly increasing trees. 
 
 \subsection{Increasing trees and Jacobi elliptic functions}The context-free grammar  was introduced by  Chen in~\cite{Chen1993} and has been found useful in studying various combinatorial structures~\cite{Chen1993,Chen2017,Dumont1996,Ma2019}, including permutations, increasing trees, labeled rooted trees and set partitions.
 
  Let $V=\{x,y,z,\ldots\}$ be a set of commutative variables. A {\em context-free grammar} $G$ over $V$ is a set of substitution rules that replace a variable in $V$ by a Laurent polynomial with variables in $V$. The formal derivative $D$ associated with a context-free grammar $G$  is defined by $D(x)=G(x)$ for any $x\in V$ and obeys the relations:
\begin{align*}
D(u+v)=D(u)+D(v)\quad\text{and}\quad D(uv)=D(u)v+uD(v),
\end{align*}
where $u$ and $v$ are two Laurent polynomials of variables in $V$. 
For example, if $V=\{x,y\}$ and 
\begin{equation}\label{gra:euler}
G=\{x\rightarrow xy, y\rightarrow xy\},
\end{equation}
 then $D(x)=xy$, $D^2(x)=xy(x+y)$ and $D^3(x)=D(xy)(x+y)+D(x+y)xy=x^3y+4x^2y^2+xy^3$. This is the grammar introduced by Dumont~\cite{Dumont1996} to generate the {\em bivariate Eulerian polynomials}.

 \begin{theorem}\label{thm:schett}
 Let $D$ be the formal derivative associated with the grammar
\begin{equation}\label{gra:schett}
G=\{x\rightarrow yz, y\rightarrow xz, z\rightarrow xy\}.
\end{equation}
Then 
\begin{equation}\label{tree:schett}
D^n(x)=\sum_{T\in\mathcal{T}_n}x^{\ee(T)}y^{\oe(T)}z^{\o(T)}.
\end{equation}
Consequently, Theorem~\ref{int:schett} is true. 
 \end{theorem}
 \begin{figure}
\includegraphics[width=5cm,height=4.5cm]{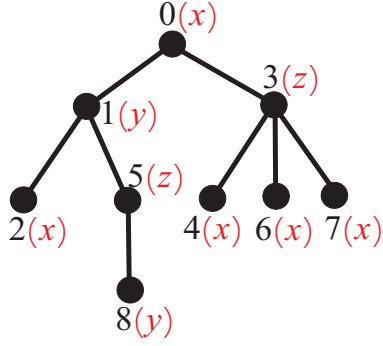}\caption{A labeling of an increasing tree.}\label{gram:tree}
\end{figure}
 \begin{proof}
 For an increasing tree $T\in\T_m$, we introduce a grammatical labeling of all nodes in $T$ by  variables from $V=\{x,y,z\}$ as follows:
\begin{itemize}
\item[(L1)] If $v$ is an even-degree node  on even level, then label $v$ by $x$; 
\item[(L2)] If $v$ is an even-degree node  on odd level, then label $v$ by $y$; 
\item[(L3)] If $v$ is an odd-degree node, then label $v$ by $z$. 
\end{itemize}
 See Fig.~\ref{gram:tree} for a labeling of an increasing tree.  It is clear that the weight $x^{\ee(T)}y^{\oe(T)}z^{\o(T)}$ for the tree $T$ equals the product of all the labeling in $T$. 
 
 We proceed to show~\eqref{tree:schett} by induction on $n$. The identity is obviously true for the initial case $n=0$, as $\T_0$ contains only one tree with a single node. Note that any increasing tree in $\T_n$ can be constructed from a unique tree from $\T_{n-1}$ by attaching the new node $n$. For a tree $T\in\T_{n-1}$ and a node $v$ of $T$, let $T'$ be the tree obtained from $T$ by attaching the new node $n$ to $v$. We have the following three cases according to the weight of $v$:
 \begin{itemize}
 \item If the node $v$ is weighted $x$ in $T$, then the weight of $v$ in $T'$ becomes $z$, while the weight of node $n$ is $y$. This corresponds to applying the rule $x\rightarrow yz$ to the label $x$ associated with $v$. 
 \item If the node $v$ is weighted $y$ in $T$, then the weight of $v$ in $T'$ becomes $z$, while the weight of node $n$ is $x$. This corresponds to applying the rule $y\rightarrow xz$ to the label $y$ associated with $v$. 
 \item If the node $v$ is weighted $z$ in $T$ and on odd level (resp.~even level), then the weight of $v$ in $T'$ becomes $x$ (resp.~$y$), while the weight of node $n$ is $y$ (resp.~$x$). This corresponds to applying the rule $z\rightarrow xy$ to the label $z$ associated with $v$. 
 \end{itemize}
 Hence the action of the formal derivative $D$ on the set of weights of trees in $\T_{n-1}$ gives the set of weights of trees in $\T_n$. This proves ~\eqref{tree:schett} by induction. 
\end{proof} 
 
 \begin{remark}
 As the grammar $G$ in~\eqref{gra:schett} is symmetry in $y$ and $z$, the polynomial $D^n(x)$ is symmetry in  $y$ and $z$. In view of~\eqref{tree:schett}, this provides a context-free grammar approach to Liu's refined symmetry 
 $$
 \sum_{T\in\mathcal{T}_n}x^{\ee(T)}y^{\oe(T)}z^{\o(T)}= \sum_{T\in\mathcal{T}_n}x^{\ee(T)}y^{\o(T)}z^{\oe(T)}.
 $$
 \end{remark}
 
  \begin{theorem}\label{thm:schett2}
 Let $D$ be the formal derivative associated with the grammar
\begin{equation}\label{gra:schett2}
G=\{w\rightarrow wy, x\rightarrow yz, y\rightarrow xz, z\rightarrow xy\}.
\end{equation}
Then 
\begin{equation}\label{tree:schett2}
D^n(w)=w\sum_{T\in\mathcal{T}_n}x^{\ee^*(T)}y^{\oe^*(T)}z^{\o^*(T)}.
\end{equation}
 \end{theorem}
 \begin{proof}
 The proof  is almost identical to that of Theorem~\ref{thm:schett} but using a slightly different labeling of trees: the root is labeled $w$ and other nodes are labeled according to (L1), (L2) and (L3) as in the proof of  Theorem~\ref{thm:schett}. The straightforward details of the discussions are omitted due to the similarity. 
 \end{proof}
 
 \begin{remark}
 Since the grammar $G$ in~\eqref{gra:schett2} is symmetry in $x$ and $z$, Theorem~\ref{thm:schett2} provides a context-free grammar approach to the refined symmetry for increasing trees in Corollary~\ref{weak-nonzero}:
 $$\sum_{T\in\mathcal{T}_n}x^{\o^*(T)}y^{\oe^*(T)}z^{\ee^*(T)}=\sum_{T\in\mathcal{T}_n}x^{\ee^*(T)}y^{\oe^*(T)}z^{\o^*(T)}.$$
 \end{remark}
 
 As another  application of Theorems~\ref{thm:schett2} and~\ref{thm:schett}, we can confirm affirmatively  a conjecture posed by Ma--Mansour--Wang--Yeh~\cite{Ma2019}. Actually, the context-free grammar $G$ in~\eqref{gra:schett2} was first considered in~\cite{Ma2019}, as a conjunction of the two grammars in~\eqref{gra:euler} and~\eqref{gra:schett}. Define the integer sequences $t_{n,i,j}$ by 
 \begin{align*}
D^{2m}(w)&=w\sum_{i,j\geq0} t_{2m,i,j}x^{i}y^{2j}z^{2m-2i-j}, \\
D^{2m+1}(w)&=w\sum_{i,j\geq0} t_{2m+1,i,j}x^{i}y^{2j+1}z^{2m-2i-j}.
\end{align*} 
 By~\eqref{tree:schett2}, we have the interpretation for $ t_{m,i,j}$ as 
 \begin{equation}\label{int:tn}
 t_{n,i,j}=|\{T\in\T_n: \ee^*(T)=i,\lfloor\oe(T)/2\rfloor=j\}|. 
 \end{equation}
 On the other side, it follows from~\eqref{tree:schett} that 
  \begin{equation}\label{int:sn}
 s_{n,i,j}=|\{T\in\T_n: \lfloor\ee(T)/2\rfloor=i,\lfloor\oe(T)/2\rfloor=j\}|. 
 \end{equation}
 In addition, observe that for $m\geq1$:
 \begin{itemize}
  \item an increasing tree $T\in\T_{2m-1}$ with $\ee(T)$ being  even (resp.~odd) must have odd-degree (resp.~even-degree) root;
 \item an increasing tree $T\in\T_{2m}$ with $\ee(T)$ being  even (resp.~odd) must have even-degree (resp.~odd-degree) root.
 \end{itemize}
 This observation can be checked easily by induction on $m$. Therefore, comparing~\eqref{int:tn} with~\eqref{int:sn} leads to the relationships:
 \begin{align*}
 s_{2m-1,i,j}&=t_{2m-1,2i-1,j}+t_{2m-1,2i,j},\\
  s_{2m,i,j}&=t_{2m,2i+1,j}+t_{2m,2i,j}
 \end{align*}
 for $m\geq1$ and $i,j\geq0$. This confirms a conjecture posed at the end of~\cite{Ma2019}.

 \subsection{Plane trees and a generating function proof of~\eqref{Weak:thm}}%
 The rest of this section is mainly  devoted to a generating function proof of~\eqref{Weak:thm}.
 
 For a weakly increasing tree $T$, the number of odd-degree nodes in even (resp.~odd) levels of $T$ is denoted by $\eo(T)$ (resp.~$\oo(T)$).
 We begin with the case of plane trees. Let 
 \begin{align}\label{def:N}
 N&=N(x,y,z,w;t):=\sum_{n\geq0}t^n\sum_{T\in\mathcal{P}_n}x^{\oe(T)}y^{\ee(T)}z^{\oo(T)}w^{\eo(T)}\\
 &=y+wxt+(wyz+x^2y)t^2+(w^2xz+wx^3+wxy^2+2xy^2z)t^3+\cdots\nonumber
 \end{align}
 and let 
 \begin{equation}\label{def:Nstar}
 N^*=N^*(x,y,z,w;t):=N(y,x,w,z;t). 
 \end{equation}
Using the first decomposition of plane trees in Fig.~\ref{decom:plan}, we obtain the system of functional equations
  \begin{figure}
\includegraphics[width=8cm,height=3cm]{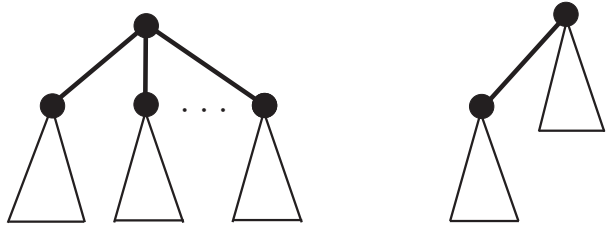}\caption{Two decompositions of plane trees.}\label{decom:plan}
\end{figure}
 \begin{equation}\label{func:N}
 N=\frac{y+wtN^*}{1-(tN^*)^2},
 \end{equation}
 \begin{equation}\label{func:Nstar}
 N^*=\frac{x+ztN}{1-(tN)^2}.
 \end{equation}
 Eliminating $N^*$ we get 
\begin{multline}\label{alg:N}
 t^4N^5-t^4yN^4-2t^2N^3+2t^2yN^2+N-y-twx\\
 +t^4(wz-z^2)N^3+t^3(wx-2xz)N^2-t^2(wz+x^2)N=0.
 \end{multline}
 Setting $w=z$ we have (with $N=N(x,y,z,z;t)$)
 \begin{equation}\label{alg:Nw=z}
 t^4N^5-t^4yN^4-2t^2N^3+(2t^2y-t^3xz)N^2+N-t^2(z^2+x^2)N-y-tzx=0.
 \end{equation}
 It follows from the above functional equation  that $N(x,y,z,z;t)$ is symmetric in $x$ and $z$, which proves algebraically the refined symmetry~\eqref{Weak:thm} for plane trees.

 Next we consider the symmetry for weakly increasing trees on two letters, i.e., on the multiset $\{1^{p_1},2^{p_2}\}$. Let 
 $$
 N^{(2)}=N(x,y,z,w;t_1,t_2):=\sum_{p_1,p_2\geq0}t_1^{p_1}t_2^{p_2}\sum_{T\in\T_{\{1^{p_1},2^{p_2}\}}}x^{\oe(T)}y^{\ee(T)}z^{\oo(T)}w^{\eo(T)}. 
 $$
 Since every weakly increasing trees on $\{1^{p_1},2^{p_2}\}$ can be obtained from a unique weakly increasing trees on $\{1^{p_1}\}$ (i.e., plane trees with $p_1$ nodes labeled by $1$) by attaching to each node a certain  plane tree (with nodes labeled by $2$), which in terms of generating functions (taking the four statistics into account) asserts that  $N(x,y,z,w;t_1,t_2)$ can be obtained from $N(x,y,z,w;t_1)$ by performing the following substitutions: 
 \begin{align}\label{substi:x}
 &x\leftarrow \frac{x+zt_2N}{1-(t_2N)^2}=N^*,\\
 \label{substi:y} &y\leftarrow \frac{y+wt_2N^*}{1-(t_2N^*)^2}=N,\\
 \label{substi:z} &z\leftarrow \frac{z+xt_2N}{1-(t_2N)^2}=\frac{N^*-x}{t_2N},\\
 \label{substi:w} &w\leftarrow \frac{w+yt_2N^*}{1-(t_2N^*)^2}=\frac{N-y}{t_2N^*},
 \end{align}
 where $N=N(x,y,z,w;t_2)$ and $N^*=N(y,x,w,z;t_2)$. Note that the equalities~\eqref{substi:x} and~\eqref{substi:y} follows from~\eqref{func:Nstar} and~\eqref{func:N}, respectively.     To see the equality in~\eqref{substi:w}, notice that every plane trees with at least one edge has the decomposition depicted in Fig.~\ref{decom:plan} (right side), which in terms of generating functions asserts that $N-y$ can be obtained from $wxt_2$ by performing the  substitutions $x\leftarrow N^*$ and $w\leftarrow \frac{w+yt_2N^*}{1-(t_2N^*)^2}$. Thus 
 $$
\frac{t_2N^*(w+yt_2N^*)}{1-(t_2N^*)^2}=N-y,
 $$
 which proves the equality in~\eqref{substi:w}.  The equality in~\eqref{substi:z} then follows from that in~\eqref{substi:w} by exchanging the variables $x\leftrightarrow y$ and $z\leftrightarrow w$. 
 Now the key observation is that $w-z$ after the substitutions~\eqref{substi:z} and~\eqref{substi:w} becomes 
 \begin{equation}\label{w-z}
 \frac{N-y}{t_2N^*}-\frac{N^*-x}{t_2N}=\frac{N-y}{t_2N^*}-\frac{N^*-x}{t_2N}+\frac{F(t_2)}{t_2(1-t_2N)(1+t_2N)(x+t_2zN)}=w-z,
 \end{equation}
 where $F(t)$ (must vanish) denotes the right-hand side of~\eqref{alg:N}  and the last equality follows from  relationship~\eqref{func:Nstar} by calculations (using Maple). That is, $w-z$ is invariant under the substitutions~\eqref{substi:z} and~\eqref{substi:w}, magically. Moreover, after the substitutions~\eqref{substi:x} and~\eqref{substi:z}, $xz$ becomes 
 \begin{equation}\label{xz}
 \frac{N^*(N^*-x)}{t_2N}=\frac{(1+(t_2N)^2)xz+t_2N(x^2+z^2)}{(1-t_2N)^2(1+t_2N)^2},
 \end{equation}
 while $x+z$ becomes 
 \begin{equation}\label{x+z}
 N^*+\frac{N^*-x}{t_2N}=\frac{x+z}{1-t_2N}.
 \end{equation}
 Since 
 $$wz-z^2=(w-z)z,\,\,wx-2xz=(w-z)x-xz\,\,\text{ and }\,\, wz+x^2=(w-z)z+(z^2+x^2),
 $$
 it follows from~\eqref{alg:N},~\eqref{w-z},~\eqref{xz} and~\eqref{x+z} that $N^{(2)}$ satisfies the functional equation 
 $$
 F(N^{(2)},N,y,xz,x+z)+(w-z)\bar{F}=0,
 $$
 where $F$ is a polynomial in   five variables with coefficients in $\Z[t_1,t_2]$ and $\bar{F}$ is a formal power series in $\Z[x,y,z,w][[t_1,t_2]]$. This proves that $N^{(2)}|_{w=z}$ is symmetric in $x$ and $z$, as $N|_{w=z}$ is. 
 
 In general, set
 $$
 N^{(n)}=N(x,y,z,w;t_1,\ldots,t_n):=\sum_{p_1,\ldots,p_n\geq0}t_1^{p_1}\cdots t_n^{p_n}\sum_{T\in\T_{\{1^{p_1},\ldots,n^{p_n}\}}}x^{\oe(T)}y^{\ee(T)}z^{\oo(T)}w^{\eo(T)}.
 $$
 As every weakly increasing trees on $\{1^{p_1},\ldots,n^{p_n}\}$ can be obtained from a unique weakly increasing trees on $\{1^{p_1},\ldots,(n-1)^{p_{n-1}}\}$ by attaching to each node a certain  plane tree (with nodes labeled by $n$), the generating function $N^{(n)}$ can be obtained from $N^{(n-1)}$ by performing the  substitutions in~\eqref{substi:x},~\eqref{substi:y},~\eqref{substi:z} and~\eqref{substi:w} in which the variable $t_2$ is replaced by $t_n$. 
 Thus, by induction on $n$ and exactly the same discussions as in the case $n=2$ above, we can prove that $N^{(n)}$ satisfies the functional equation 
 $$
 F(N^{(n)},N^{(n-1)},\ldots,N^{(1)},y,xz,x+z)+(w-z)\bar{F}=0,
 $$
 where $N^{(n-k)}=N^{(n-k)}(x,y,z,w;t_{k+1},\ldots,t_{n})$ for $1\leq k\leq n-1$,  $F$ is a polynomial in  $n+3$ variables with coefficients in $\Z[t_1,\ldots,t_n]$ and $\bar{F}$ is a formal power series in 
 $\Z[x,y,z,w][[t_1,\ldots,t_n]]$. This implies that $N^{(n)}|_{w=z}$ is symmetric in $x$ and $z$, as  by induction the formal series $N^{(1)}|_{w=z},\ldots,N^{(n-1)}|_{w=z}$ are  symmetric in $x$ and $z$. This provides a generating function proof of~\eqref{Weak:thm}.
\subsection{The Schett polynomials for plane trees} We will report a connection between plane trees and fighting fish with marked tail studied in~\cite{Duchi2017}. The following version of the Multivariable Lagrange Inversion Formula will be used, which is reproduced from~\cite{Gessel}, for the sake of completeness.  

\begin{theorem}[See~\text{\cite[Theorem~4]{Gessel}}]\label{F:Lagran}
 Let $g_i({\bf t})$, $i=1,\ldots,m$, be formal power series in $m$ indeterminates $t_1,\ldots,t_m$. Then there exists a unique solution $F_1,\ldots,F_m$ to the system of equations 
$$
F_i=g_i(F_1,\ldots,F_m),\quad i=1,\ldots,m,
$$
and for all Laurent series $\Phi({\bf t})$ we have 
$$
\Phi(F_1,\ldots,F_m)=\sum_{{\bf n}}[{\bf t}^{\bf n}]|K({\bf t})|\Phi({\bf t}){\bf g}^{\bf n}({\bf t}),
$$
where $|K({\bf t})|$ is the determinant of the $m\times m$ matrix  
$$
K({\bf t})=\biggl(\delta_{ij}-\frac{t_i}{g_i({\bf t})}\frac{\partial g_i({\bf t})}{\partial t_j}\biggr).
$$
\end{theorem}

Recall the definitions of $N$ and $N^*$ from~\eqref{def:N} and~\eqref{def:Nstar}, respectively.  By the functional equations~\eqref{func:N} and~\eqref{func:Nstar}, we have 
\begin{equation}\label{fun:line}
N= N(tN^*)^2+twN^*+y\quad\text{and}\quad N^*=N^*(tN)^2+tzN+x.
\end{equation}
Applying the Multivariable Lagrange Inversion Formula in Theorem~\ref{F:Lagran} then results in the following expression for $N$. 

\begin{proposition} Let $N=N(x,y,z,w;t)$ be defined in~\eqref{def:N}. Then
\begin{equation}\label{eq:jacobiN}
N=\sum\limits_{n,m\geq 0}[a^{n}b^{m+1}]K (t^2ab^2+twb+y)^n(t^2ba^2+tza+x)^m,
\end{equation}
where 
$$K=xy+atyz+btwx-2a^2b^2t^3z-2a^2b^2t^3w-4a^3b^3t^4.$$
Consequently, the number of plane trees $T$ with  $\oe(T)=i$,  $\ee(T)=j$, $\oo(T)=k$ and $\eo(T)=l$ equals
\begin{equation}\label{eq:ijkl}
(i+k)\frac{(\frac{i+k+l}{2}+j-1)!(\frac{j+k+l-1}{2}+i-1)!}{i!j!k!l!(\frac{i+k-l}{2})!(\frac{j+l-k-1}{2})!},
\end{equation}
if both $\frac{i+k-l}{2}$ and $\frac{j+l-k-1}{2}$ are non-negative integers; and $0$, otherwise.
\end{proposition}
\begin{proof}
Setting $t=1$ in the expression~\eqref{eq:jacobiN} for $N$, we have 
 \begin{align*}
 [x^iy^jz^kw^l]N|_{t=1}&={\frac{1}{2}(i+k-l)+l+j-1\choose{\frac{1}{2}(i+k-l),l,j-1}}{\frac{1}{2}(j+l-k-1)+k+i-1\choose{\frac{1}{2}(j+l-k-1),k,i-1}}\\
 &\quad+{\frac{1}{2}(i+k-l)+l+j-1\choose{\frac{1}{2}(i+k-l),l,j-1}}{\frac{1}{2}(j+l-k-1)+k+i-1\choose{\frac{1}{2}(j+l-k-1),k-1,i}}\\
 &\quad+{\frac{1}{2}(i+k-l)+l+j-1\choose{\frac{1}{2}(i+k-l),l-1,j}}{\frac{1}{2}(j+l-k-1)+k+i-1\choose{\frac{1}{2}(j+l-k-1),k,i-1}}\\
 &\quad-2{\frac{1}{2}(i+k-l-2)+l+j\choose{\frac{1}{2}(i+k-l-2),l,j}}{\frac{1}{2}(j+l-k-1)+k+i-1\choose{\frac{1}{2}(j+l-k-1),k-1,i}}\\
 &\quad-2{\frac{1}{2}(i+k-l)+l+j-1\choose{\frac{1}{2}(i+k-l),l-1,j}}{\frac{1}{2}(j+l-k-3)+k+i\choose{\frac{1}{2}(j+l-k-3),k,i}}\\
 &\quad-4{\frac{1}{2}(i+k-l-2)+l+j\choose{\frac{1}{2}(i+k-l-2),l,j}}{\frac{1}{2}(j+l-k-3)+k+i\choose{\frac{1}{2}(j+l-k-3),k,i}},
\end{align*}
which is simplified to~\eqref{eq:ijkl}. 
\end{proof}

\begin{corollary}[Jacobi elliptic functions for plane trees -- analog of $s_{2n,i,0}$]\label{cor:jaco2}
The number of plane trees $T$ with $\o(T)=0$,  $\oe(T)=2i$  and $\ee(T)=2j+1$ is 
\begin{equation}\label{eq:jaco2}
\frac{1}{2j+1}{2j+i\choose i}{2i+j-1\choose j}. 
\end{equation}
Consequently, we have the following combinatorial identity 
\begin{equation}\label{com:ternary}
\frac{1}{2n+1}{3n\choose n}=\sum_{j=0}^{n-1}\frac{1}{2j+1}{n+j\choose 2j}{2n-j-1\choose j}.
\end{equation}
\end{corollary}
\begin{proof}
Setting $t=1$, $z=w=0$ in the expression~\eqref{eq:jacobiN} for $N$, we have
$$
N(x,y,0,0;1)=\sum\limits_{n,m\geq 0}[a^{n}b^{m+1}]K (ab^2+y)^n(ba^2+x)^m,
$$
where $K=xy-4a^3b^3$.
Extracting the coefficients of $x^{2i}y^{2j+1}$ gives 
\begin{equation*}
[x^{2i}y^{2j+1}]N(x,y,0,0;1)={2j+i\choose 2j}{2i+j-1\choose 2i-1}-4{2j+i\choose 2j+1}{2i+j-1\choose 2i},
\end{equation*}
which is simplified to~\eqref{eq:jaco2}.  The combinatorial identity~\eqref{com:ternary} then follows from the fact (see~\cite{Deutsch2002}) that the number of plane trees $T$ with $2n$ edges and $\o(T)=0$ is $1/(2n+1){3n\choose n}$, a number that also enumerates ternary trees with $n$ internal nodes. 
\end{proof}

\begin{corollary}[Jacobi elliptic functions for plane trees -- analog of $s_{2n+1,0,j}$]\label{cor:fish}
The number of plane trees $T$ with $\ee(T)=0$,  $\oe(T)=2i+1$  and $\o(T)=2j+1$ is 
\begin{equation}\label{eq:fish}
\frac{2i+2j+1}{(2i+1)(2j+1)}{2i+j\choose j}{2j+i\choose i}. 
\end{equation}
\end{corollary}
\begin{proof}
Setting $t=1$, $y=0$ and $z=w$ in the expression~\eqref{eq:jacobiN} for $N$, we have
$$
N(x,0,z,z;1)=\sum\limits_{n,m\geq 0}[a^{n}b^{m+1}]K (ab^2+zb)^n(ba^2+za+x)^m,
$$
where $K=bxz-4a^2b^2z-4a^3b^3$.
Extracting the coefficients of $x^{2i+1}z^{2j+1}$ gives 
\begin{align*}
&\quad[x^{2i+1}z^{2j+1}]N(x,0,z,z;1)\\
&={2i+j\choose j}{2j+i\choose i}-4{2i+j\choose j-1}{2j+i-1\choose i-1}-4{2i+j\choose j-1}{2j+i-1\choose i-2}\\
&={2i+j\choose j}{2j+i\choose i}-4{2i+j\choose j-1}{2j+i\choose i-1},
\end{align*}
which is simplified to~\eqref{eq:fish}.  
\end{proof}

 Fighting fish were introduced by
Duchi, Guerrini, Rinaldi and Schaeffer~\cite{Duchi2017} as combinatorial structures made of square tiles that form two dimensional branching surfaces.  They showed that 
 fighting fish with $i+1$ left lower free edges and $j+1$ right lower free edges with a marked tail is also enumerated by
$$
\frac{2i+2j+1}{(2i+1)(2j+1)}{2i+j\choose j}{2j+i\choose i}. 
$$
In view of Corollary~\ref{cor:fish}, we have the following result.

\begin{theorem}\label{thm:fish}
The number of plane trees $T$ with $\ee(T)=0$,  $\oe(T)=2i+1$  and $\o(T)=2j+1$ equals the number of  fighting fish with $i+1$ left lower free edges and $j+1$ right lower free edges with a marked tail.
\end{theorem}

It would be interesting to find a bijective proof of Theorem~\ref{thm:fish}. 

\section{A group action on weakly increasing binary trees}
\label{sec:4}
 
 This section is devoted to construct a $\Z_2^{p}$-action on trees, called {\em triangle group action} (see Fig.~\ref{fig:involution}), to prove Theorem~\ref{thm:action}. We will define this group action on a class of trees, called weakly increasing binary trees which are in natural bijection with weakly increasing trees, rather than on weakly increasing trees directly,  for convenience's sake.
 
 \begin{definition}[Weakly increasing binary trees]
 A  weakly increasing binary  tree on $M$ is a  binary tree  such that 
 \begin{enumerate}[(i)]
\item  it contains $p+1$ nodes that are labeled by elements from the multiset $M\cup\{0\}$,
\item  the node $0$ has exactly one left child, and
\item  each sequence of labels along a path from the root to any leaf is weakly increasing. 
\end{enumerate}   Denote by $\mathcal{B}_{M}$ the set of weakly increasing binary trees on $M$. See Fig.~\ref{18wibt} for all weakly increasing binary trees on  $M=\{1^2,2^2\}$.
\end{definition}
 
 \begin{figure}
	\begin{center}
	\includegraphics[width=8cm,height=6.5cm]{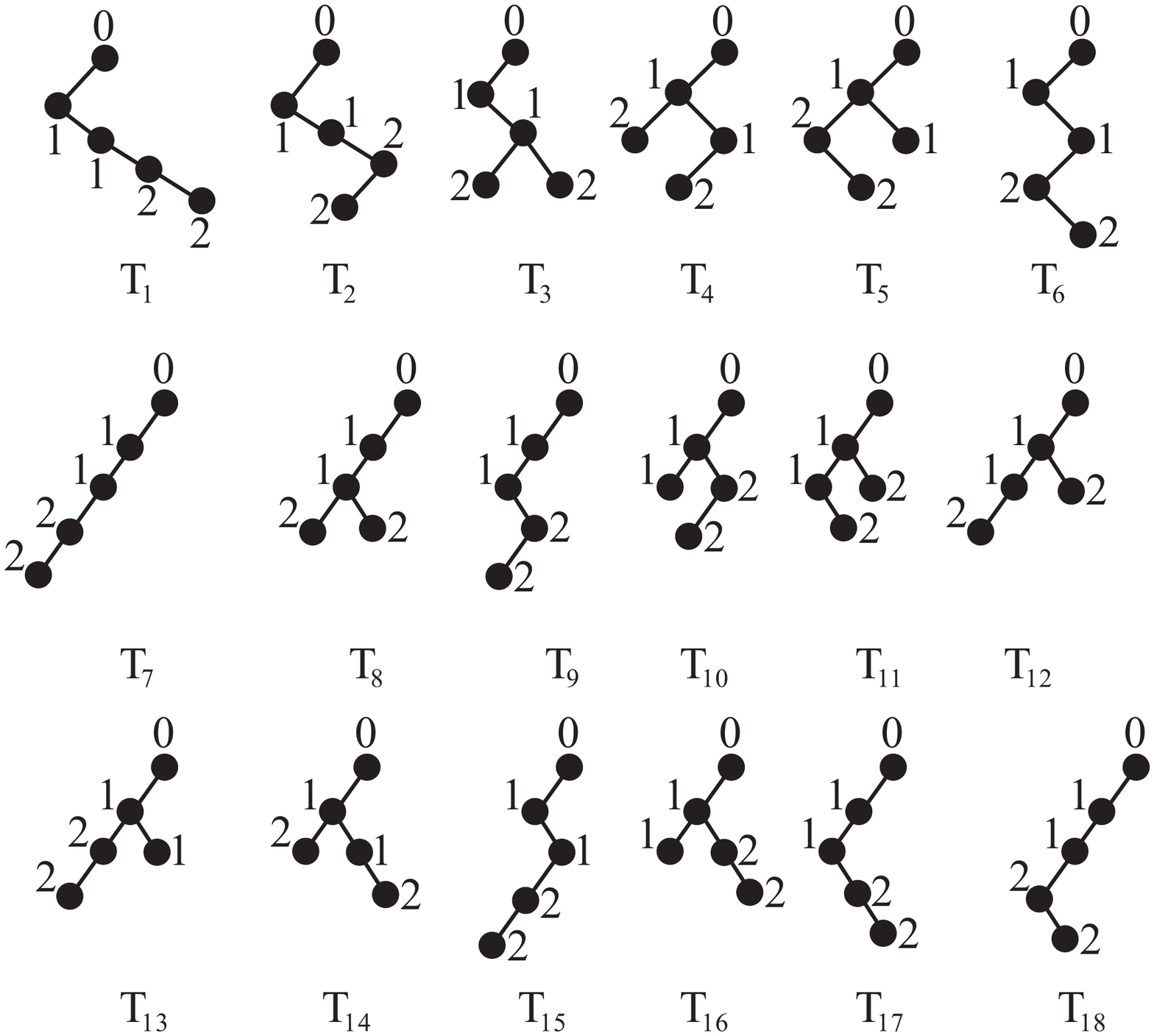}
	\end{center}
	\caption{\label{18wibt}All $18$ weakly increasing binary trees on $\{1^2,2^2\}$.}
	\end{figure}
 
 There is a natural bijection $\rho: \T_M\rightarrow \B_M$ that transforms a tree $T\in\T_M$ to the binary tree $\rho(T)$ by requiring 
 \begin{enumerate}
 \item if the node $y$ is the leftmost child of its parent $x$ in $T$, then  $y$ becomes the left child of $x$ in $\rho(T)$;
 \item if the node $y$ is the first brother to the right of $x$ in $T$, then  $y$ becomes the right child of $x$ in $\rho(T)$.
 \end{enumerate}
 Under this bijection, all trees in Fig.~\ref{18wit} are in one-to-one correspondence with the trees with the same names in Fig.~\ref{18wibt}. 
 
 In order to translate the concerned statistics for weakly increasing trees to weakly increasing binary trees, we need to introduce some terminologies. Fix a  tree $B\in \mathcal{B}_{M}$. For any node $u$, suppose that the node $v$ is the parent of $u$ in $B$. If $u$ is the left child of $v$, then the edge $uv$ connecting $u$ to $v$ is called  a {\em left edge}; otherwise, we say that $uv$ is a {\em right edge}. Note that there exists a unique path $v_0v_1\ldots v_s$ from the root $0$ to $u$ in $B$, where $v_0=0$ and $v_s=u$.  The node $u$ is in {\em left-level} $k$ if the path contains $k$ left edges. Let $j$ be the (unique) index in the path such that $v_jv_{j+1}$ is a left edge and $v_{h}v_{h+1}$ (if any) is a right edge for any $j+1\leq h\leq s-1$. We say that $v_s=u$ is a {\em right grandson} of $v_j$, and $v_j$ is the {\em ancestor} of $u$. Thus, the right-degree of a node in $B$ is the number of its right grandsons. Let us introduce the following six tree statistics:
\begin{itemize}
\item The number of nodes of {\bf r}ight-{\bf deg}ree $q$ in $B$, denoted by $\rdeg_q(B)$.
\item The number of nodes of {\bf r}ight-degree $q$ in {\bf o}dd {\bf l}eft-levels of $B$, denoted by $\rol_q(B)$.
\item The number of nodes in {\bf e}ven {\bf l}eft-{\bf l}evels of $B$, denoted by $\ell(B)$.
\item The number of {\bf o}dd {\bf r}ight-{\bf d}egree nodes in $B$, denoted by $\ord(B)$.
\item The number of {\bf e}ven {\bf r}ight-degree nodes in {\bf o}dd (resp.~{\bf e}ven) {\bf l}eft-levels of $B$, denoted by $\oler(B)$ (resp.~$\eler(B)$).
\end{itemize}
 Since the degree (resp.~level) of a node $u$ in $T$ equals the right-degree (resp.~left-level) of $u$ in $\rho(T)$, we have the following result. 
 
 \begin{lemma} \label{tree-binary}
 For any multiset $M$, 
$$
(\deg_q,\od_q,\el,\o,\oe,\ee)(T)=(\rdeg_q,\rol_q,\ell,\ord,\oler,\eler)(\rho(T))\text{ for each $T\in\T_M$}.
$$
\end{lemma}

We shall call a node $u$ in a binary tree $B\in\B_M$ {\em active} if $u$ is active in $\rho^{-1}(B)$. Alternatively, active nodes in binary trees have the following direct description. 
\begin{figure}
\begin{center}
\includegraphics[width=8.5cm,height=5cm]{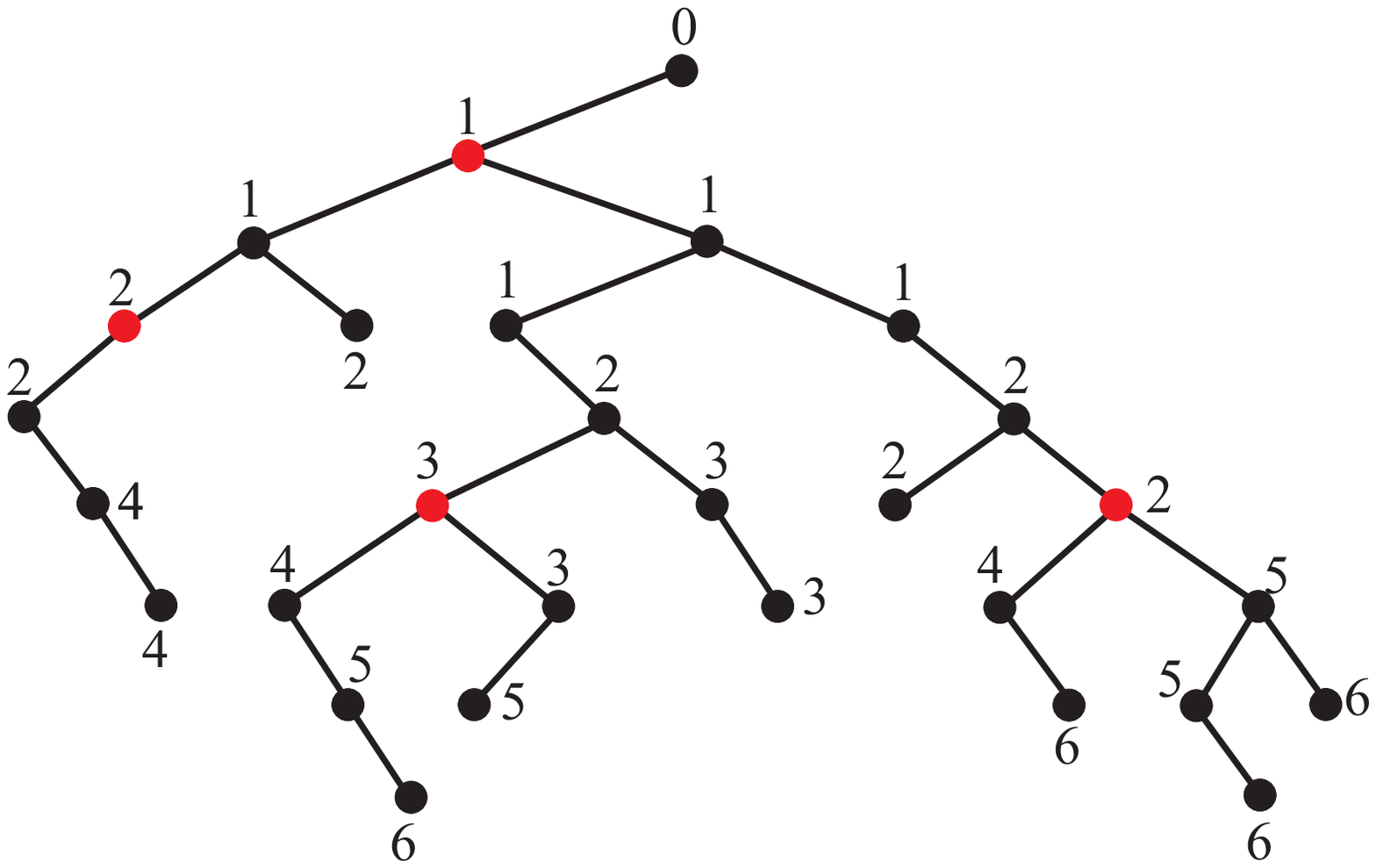}
\end{center}
\caption{\label{fig:EXA}A weakly increasing binary tree and its active nodes.}
\end{figure}
\begin{definition}[Active nodes in binary trees]
 Let $B\in\mathcal{T}_{M}$ and let $u$ be a node of $B$.
 Suppose that  $y$ (possibly empty) is the right child of $u$.  The node $u$ is {\bf\em active} if it satisfies the following three conditions:
\begin{enumerate}[(a)]
\item  the  left-level  of $u$ is odd;
\item the unique path from  $u$ to  the ancestor of $u$  has odd number of edges;
\item either (i) $y$ is not empty and the right-degrees of $u$ and $y$ have the same parity, or (ii) $y$ is empty  and the right-degree of $u$  is odd.
\end{enumerate}
 Furthermore, $u$ is said to be {\bf\em active odd} or {\bf\em active even} according to the parity of the right-degree of $u$. Let $\act(T)$ (resp.~$\eact(T)$, $\oact(T)$) be the number of active (resp.~active even, active odd) nodes of $T$. 
\end{definition}
\begin{example} Consider the weakly increasing binary tree in Fig.~\ref{fig:EXA}.  All the active nodes are drawn in red. We have $\act(B)=4$ and $\oact(B)=\eact(B)=2$. 
\end{example}

\begin{figure}
\begin{center}
\includegraphics[width=10cm,height=5cm]{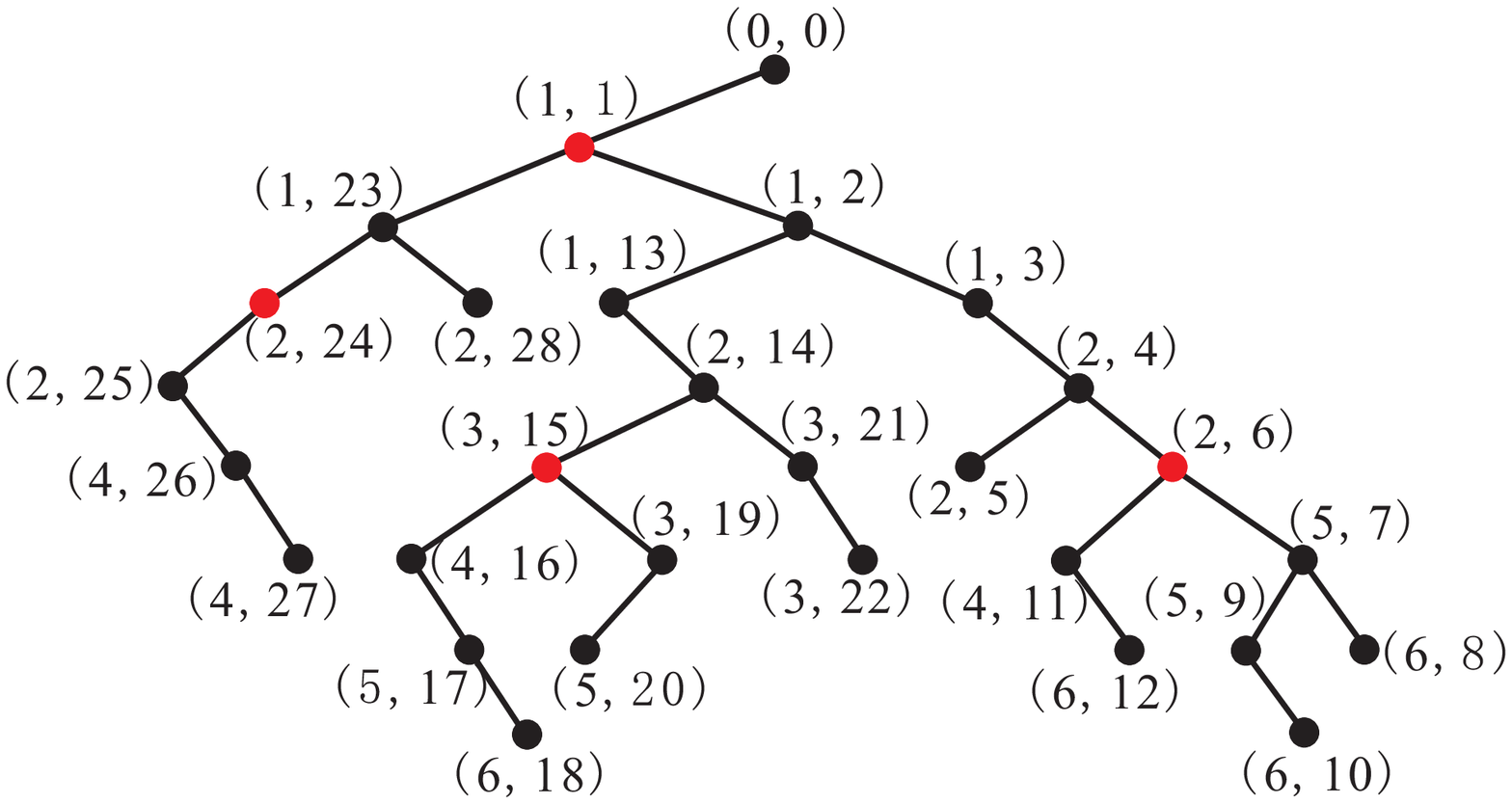}
\end{center}
\caption{\label{fig:SEARCH} An example of the modified preorder.}
\end{figure}
In order to define a group action on $\mathcal{B}_M$, we need a specified order, which is a {\em variation of  the depth-first order (or preorder)}, for the nodes of a weakly increasing binary tree.  
This specified order is called {\em modified preorder} and will be defined according to the following algorithm (here $v_i$ denotes the label of the node which is visited at the time $i$):  
\begin{itemize}
\item Initialization: Start from the root of a given  binary tree $B\in \mathcal{B}_M$, i.e.,  set  $v_0=0$.
\item  At the time $i\geq 1$, suppose that the sequence $v_0,v_1,\ldots, v_{i-1}$ of nodes  are visited. Let $k$ be the largest index such that at least one child of the node $v_k$ has not been visited.
If $v_k$ has exactly one child $u$ that has not been visited, then we set $v_i=u$. Otherwise, the left child $x$ and the right child $y$ of $v_k$ have not been visited. We need to choose $x$ or $y$ for the next visit according to the following three cases: 
\begin{enumerate}
  \item  $v_k$ is active. Then set 
  $$
  v_i=
  \begin{cases}
  y,  \quad&\text{if the right-degree of $v_k$ is even;}\\
  x, \quad& \text{otherwise.}
  \end{cases}
  $$
 \item  $v_k$ is not active, but its parent $w$ is active.  Then set 
  $$
  v_i=
  \begin{cases}
  y,  \quad&\text{if $v_k$ is the right child of $w$;}\\
  x, \quad& \text{otherwise.}
  \end{cases}
  $$
\item neither $v_k$ nor its parent is  active.  Then set  $v_i=x$.
\end{enumerate}
\item Iterating step~2 above until $i=p$, we obtain the desired order $v_0,v_1,\ldots, v_{p}$ of nodes in $B$. 
\end{itemize}
Moreover, if  $u$  is the $i$-th node in the above modified preorder, we simply say that $u$ is the $i$-th node of the tree and can relabel this node with a pair  $(u,i)$.
For example, the modified preorder for the tree in Fig.~\ref{fig:EXA} is drawn  in Fig.~\ref{fig:SEARCH}.

Now we can define our fundamental transformation on weakly increasing binary trees. For $B\in\mathcal{B}_M$ and $i\in[p]$, suppose that the $i$-th node  of $B$ is  $u$.
Then we define a weakly increasing binary tree $\Lambda_i(B)$ as follows:
\begin{itemize}
\item If $u$ is not active, then $\Lambda_i(B)=B$.
\item  If $u$ is active, then suppose that the nodes $x$ and $y$ are the left and right child of $u$, $F$ and $A$  are the subtrees rooted at the left and right children of $x$, $D$ and $H$ are the subtrees rooted at the left and right children of $y$, respectively. Note that the nodes $x, y$ and the subtrees $F, A, D, H$ may be empty.  Delete the edges $ux$ and $uy$, remove the edges between $x$ (resp.~$y$) and its children. Attach $x$ to be the right child of $u$, $y$ to be the left child $u$, $F$ (resp.~$A$) to be the right (resp.~left) subtree of $x$,  $H$ (resp.~$D$) to be the left (resp.~right) subtree of $x$. Denote by $\Lambda_i(B)$ the resulting weakly increasing binary tree. Equivalently, $\Lambda_i(B)$ is obtained from $B$ by switching  the left and  right branches at the three nodes: $u$, $x$ and $y$. 
The construction of  $\Lambda_i(B)$ has a nice visualization as depicted in Fig.~\ref{fig:involution}, where $u$ is the left child of its parent.
\begin{figure}
\begin{center}
\includegraphics[width=8cm,height=4cm]{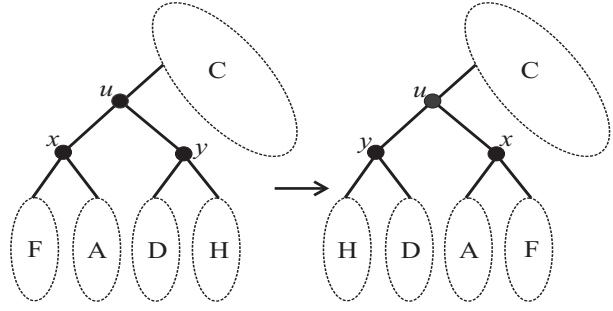}
\end{center}
\caption{\label{fig:involution}The transformation $B\mapsto\Lambda_i(B)$.}
\end{figure}
\end{itemize}

The transformation $\Lambda_i$ has the following properties. The first two can be verified from the construction of $\Lambda_i$ directly. 

\begin{lemma}\label{LEM:order} 
Let $B\in\mathcal{B}_M$ and $i\in [p]$. Then
 for any $j\in [p]$,
  $u$ is the $j$-th node of $B$  if and only if  $u$ is the $j$-th node of $\Lambda_i(B)$. In other words, the modified preorder is invariant under the transformation $\Lambda_i$.
\end{lemma}

\begin{lemma}\label{LEM:bi-even->odd}
 Let $B\in\mathcal{B}_M$ and $i\in[p]$. If $u$ is the $i$-th node of $B$, then $u$ is an active even (resp.~odd) node of $B$ if and only if $u$ is an active odd (resp.~even) node of $\Lambda_i(B)$. Moreover, $\Lambda_i$ is an involution on $\mathcal{B}_M$.
\end{lemma}

\begin{lemma}\label{LEM:bi-eler-action} Let $B\in\mathcal{B}_M$   and  $i\in [p]$.
Then $\mathsf{eler}(B)=\mathsf{eler}(\Lambda_i(B))$.
\end{lemma}
\begin{proof} 
If the $i$-th node $u$ in $B$ is not active, then $\Lambda_i(B)=B$ and so $\eler(B)=\eler(\Lambda_i(B))$. Suppose that $u$ is  active in $B$ with ancestor $w$, left child $x$ and   right child $y$ (possibly empty). Since $\Lambda_i$ preserves the parity of left-level and  right-degree of all other nodes except for the four ones $u$, $w$, $x$ and $y$, we only need to focus on the properties of these four nodes.
As $x$ or $y$ may be empty, we distinguish the following three cases:
\begin{enumerate}
\item Neither $x$ nor $y$ is empty in $B$. As $u$ is active, $u$ and $y$ are in odd left-level of  $B$, and  also $u$ and $x$ are in odd left-level of   $\Lambda_i(B)$. Moreover, the left-level and  right-degree of  $w$ (resp.~$x$) have the same parity as those of node $y$ (resp.~$w$) in $\Lambda_i(B)$. This implies that $\mathsf{eler}(B)=\mathsf{eler}(\Lambda_i(B))$.

\item Only  $x$ is not empty in $B$. As $u$ is active, $u$ is in odd left-level of  $B$, and  also $u$ and $x$ are in odd left-level of  $\Lambda_i(B)$.  Moreover, $w$ has right-degree odd, and the left-level and  right-degree of the node  $x$  in $B$ have the same parity as those of node $w$  in $\Lambda_i(B)$. Hence, $\mathsf{eler}(B)=\mathsf{eler}(\Lambda_i(B))$.

\item Only  $y$ is not empty in $B$.   As $u$ is active, $u$ and $y$ are in odd left-level of  $B$, and  also $u$ is in odd left-level of  $\Lambda_i(B)$ and $w$ has right-degree one in $\Lambda_i(B)$.  Moreover,  the left-level and  right-degree of the node  $w$  in $B$ have the same parity as those of node $y$   in $\Lambda_i(B)$. Hence, $\mathsf{eler}(B)=\mathsf{eler}(\Lambda_i(B))$.
\end{enumerate}
In either case, we have $\mathsf{eler}(B)=\mathsf{eler}(\Lambda_i(B))$.
\end{proof}

\begin{definition}\label{DEF:active-odd-even}
Let $u$ be an active node in a binary tree $B$.
If $u$ has the right child $y$ and $u$ is active even (resp.~odd), then the nodes $u$ and $y$ are said to be {\bf\em dynamic  even (resp.~odd)  nodes}.
If $u$ has only left child,  then the nodes $u$ and its  ancestor  are said to be  {\bf\em dynamic odd nodes}.
Denote by $\dme(B)$ (resp.~$\dmo(B)$) the number of dynamic  even (resp.~odd)  nodes in $B$. Clearly, 
\begin{equation}\label{dem:eact}
\dme(B)=2\cdot\eact(B)\quad \text{and} \quad \dmo(B)=2\cdot\oact(B).
\end{equation}
\end{definition}

The following property of $\Lambda_i$ is clear from its construction, which  justifies the name dynamic even or odd nodes. 
\begin{lemma}\label{LEM:dym} Let $B\in\mathcal{B}_M$ and $i\in[p]$. Suppose that  the $i$-th node $u$ of $B$ is active and the ancestor of $u$ is $w$. 
\begin{enumerate}[($a$)]
\item  If $u$ has the left child $x$ and  the right child $y$ in $B$, then  $u$ and  $y$   are dynamic even (resp.~odd) in $B$ if and only if  $u$ and  $x$  are dynamic odd (resp.~even) in $\Lambda_i(B)$.
\item  If  $u$ has only the left child $x$ in $B$, then  $u$ and its ancestor $w$ are dynamic odd in $B$ if and only if  $u$ and  $x$  are dynamic even in $\Lambda_i(B)$.
\item  If $u$ has only the right child $y$ in $B$, then  $u$ and  $y$ are dynamic even in $B$ if and only if  $u$ and its ancestor $w$  are dynamic odd in $\Lambda_i(B)$.
 \end{enumerate}
  Moreover, the parity of left-levels and  right-degrees of all other nodes other than $u$, $w$, $x$ and $y$ remain unchanged under $\Lambda_i$.  
\end{lemma}

For a binary tree $B\in\B_M$, let $\ndoler(B)$ be the number of {\bf e}ven  {\bf r}ight-degree nodes in {\bf o}dd {\bf l}eft-levels of $B$ that are {\bf n}on-{\bf d}ynamic even and let $\ndord(B)$ be the number of {\bf o}dd {\bf r}ight-{\bf d}egree nodes of $B$ that are {\bf n}on-{\bf d}ynamic odd. For example, in the weakly increasing  binary tree in Fig.~\ref{fig:SEARCH}, the two non-dynamic even right-degree nodes are $(1,3)$ and $(6,8)$ and the two non-dynamic odd right-degree nodes are $(0,0)$ and $(2,4)$. As dynamic even nodes are in odd left-levels, we have the relationships 
\begin{equation}\label{rel:ndoler}
\dme(B)+\ndoler(B)=\oler(B)\quad\text{and}\quad \dmo(B)+\ndord(B)=\ord(B).
\end{equation}
Moreover, we have the following crucial equality. 
\begin{lemma} \label{lem:main}
For any $B\in\B_M$, $\ndoler(B)=\ndord(B)$.
\end{lemma}
\begin{proof}
 We will prove the  identity by induction on $p$. Clearly, the equality is true for $p=1$, as $\ndoler(B)=\ndord(B)=1$. For $p\geq2$, let $B'$ be the tree obtained from $B$ by removing the $p$-th node $v$, which must be a leaf, from $B$. Suppose that the parent (resp.~ancestor) of $v$ is $u$ (resp.~$w$). We discuss the following two cases:
\begin{itemize}
\item {\bf Case 1:} $v$ is the left child of $u$. We further distinguish two cases:
\begin{enumerate}[a)]
\item $v$ is in odd left-level of $B$. Since $v$ is the last node, $u$ has no right child in $B$. This implies that $\ndoler(B)=\ndoler(B')+1=\ndord(B')+1=\ndord(B)$.

\item $v$ is in even left-level of $B$. 
\begin{enumerate}[(i)]
\item If $u$ is active, then $v$ is the unique child of $u$.  By the definition of dynamic  odd nodes, $u$ and $w$ are both  dynamic  odd nodes in $B$, while in $B'$, $u$ is non-active even right-degree node in odd left-level  and $w$ is non-dynamic odd right-degree node. Hence, $\ndoler(B)=\ndoler(B')-1$ and $\ndord(B)=\ndord(B')-1$. 

\item If $u$ is not active, but the parent $u'$ of $u$ is active, then $u'$ and $u$ are both dynamic odd in $B$. But in $B'$, $u'$ is  non-dynamic odd right-degree node  and $u$  is   non-dynamic even right-degree node  in odd  left-level. Hence, $\ndoler(B)=\ndoler(B')-1$ and $\ndord(B)=\ndord(B')-1$.   

\item If  $u$ and its parent $u'$ both are not active, then $u$ is the right child of $u'$ and $v$ is the unique child of $u$. So in $B$, $u'$ is  non-dynamic even right-degree node  in odd  left-level and $u$ is non-dynamic odd right-degree node. But in $B'$, $u'$ and $u$ are both dynamic even. Thus, $\ndoler(B)=\ndoler(B')+1$ and $\ndord(B)=\ndord(B')+1$.
   \end{enumerate}
\end{enumerate}
\item {\bf Case 2:} $v$ is the right child of $u$. Suppose that $u$ is the $k$-th right grandson of $w$. We further distinguish two cases:
\begin{enumerate}[a)]
\item  $v$ is in odd left-level of $B$.
\begin{enumerate}[(i)] 
\item If $u$ is active, then $v$ is the unique child of $u$ and $k$ is odd. So in $B$, $w$ is even right-degree node in even left-level,   and $u$ and $v$ are dynamic even. But in $B'$, $w$ is non-dynamic odd right-degree node  and $u$ is non-dynamic even right-degree node in odd left-level. Thus, we have $\ndoler(B)=\ndoler(B')-1$ and $\ndord(B)=\ndord(B')-1$.  
\item If $u$ is not active and $k$ is odd, then $u$ has odd right-degree. In $B$, $w$ is even right-degree node in even left-level, $u$ is non-dynamic odd right-degree node, $v$ is non-dynamic even right-degree node in odd left-level. In $B'$, $u$ and $w$ are dynamic odd. Thus, $\ndoler(B)=\ndoler(B')+1$ and $\ndord(B)=\ndord(B')+1$. 
If $k$ is even, then in $B$, $w$ is non-dynamic  odd right-degree node and $v$ is non-dynamic  even right-degree node in odd left-level.  But in $B'$,  $w$ is  even right-degree node in even left-level. Hence, $\ndoler(B)=\ndoler(B')+1$ and $\ndord(B)=\ndord(B')+1$.
\end{enumerate}
\item  $v$ is in even left-level of $B$. 
\begin{enumerate}[(i)]
\item If   $w$ is active in $B$, then no matter $w$ is active even or active odd we have $\ndoler(B)=\ndoler(B')-1$ and $\ndord(B)=\ndord(B')-1$. 
\item If $w$ is not active and  has right child in $B$, then the parent of  $w$ must be active (as $v$ is the last node).  In this case, $\ndoler(B)=\ndoler(B')-1$ and $\ndord(B)=\ndord(B')-1$. 

\item If $w$ is not active and  has no right child in $B$, then we distinguish two cases according to $w$ is dynamic or not, i.e., the parent $w'$ of $w$ is active or not. If $w$ is dynamic in $B$, then  $w'$ is active. Thus, $w$ and $w'$ are dynamic in $B$, while in $B'$, $w$ and $w'$ are no longer dynamic and they have different parities of right-degrees. Therefore, no matter $w$ has odd or even right-degree, we have $\ndoler(B)=\ndoler(B')-1$ and $\ndord(B)=\ndord(B')-1$. Otherwise, $w$ is non-dynamic  and its parent $w'$ is non-active in $B$. Since the parities of the right-degrees of $w$ and $w'$ are different (as $w'$ is non-active) in $B$, we see that $w'$ is active in $B'$. Thus, $w$ and $w'$ becomes dynamic in $B'$, which forces $\ndoler(B)=\ndoler(B')+1$ and $\ndord(B)=\ndord(B')+1$, no matter $w$ has odd or even right-degree. 
\end{enumerate}
\end{enumerate}
\end{itemize}
 
 In conclusion, the above discussions prove the equality by induction.
\end{proof}

Now we are in position to prove Theorem~\ref{thm:action}.

\begin{proof}[{\bf Proof of Theorem~\ref{thm:action}}] In view of~\eqref{rel:reduce}, the partial $\gamma$-positivity expansion~\eqref{eq:partial-gm} of the reduced Schett polynomial $\hat S_M(x,y,z)$ is equivalent to 
\begin{equation}\label{wit:partial}
\sum_{T\in\T_M}x^{\ee(T)}y^{\oe(T)}z^{\o(T)}=\sum_{i}x^{i}\sum_{j=0}^{\frac{p+1-i}{2}}\gamma_{M,i,j}(yz)^{\frac{p+1-i}{2}-j}(y^2+z^2)^{j},
\end{equation} where $\gamma_{M,i,j}$ is the number of weakly increasing trees $T\in\mathcal{T}_M$ with $\ee(T)=i$, $\eact(T)=0$,  and $\act(T)=j$. By Lemma~\ref{tree-binary} and the definition of active nodes in weakly increasing binary trees, expansion~\eqref{wit:partial} is transformed by $\rho$ into 
\begin{equation}\label{wibt:partial}
\sum_{B\in\B_M}x^{\eler(B)}y^{\oler(B)}z^{\ord(B)}=\sum_{i}x^{i}\sum_{j=0}^{\frac{p+1-i}{2}}\gamma_{M,i,j}(yz)^{\frac{p+1-i}{2}-j}(y^2+z^2)^{j},
\end{equation} where $\gamma_{M,i,j}$ is the number of weakly increasing binary trees $B\in\mathcal{B}_M$ with $\eler(B)=i$, $\eact(B)=0$,  and $\act(B)=j$. We aim to prove expansion~\eqref{wibt:partial} by introducing a $\Z_2^{p}$-action on weakly increasing binary trees based on the fundamental transformation $\Lambda_i$. 

For any $B\in\B_M$, note that if $u$ is active, then neither its parent nor its children are active. Thus, the transformations $\Lambda_i$ are  commutative, i.e., $\Lambda_i\circ\Lambda_j=\Lambda_j\circ\Lambda_i$. Moreover,  by Lemma~\ref{LEM:bi-even->odd}, the transformations $\Lambda_i$ are also involutions. Therefore, for any subset $S\subseteq[p]$ we can then define the function $\Lambda_S:\B_M\rightarrow\B_M$ by
$$
\Lambda_S(T)=\prod_{i\in S}\Lambda_i(B).
$$
Hence  the group $\mathbb{Z}_2^p$ acts on $\B_M$ via the functions $\Lambda_S$, $S\subseteq [p]$. We call this action the {\em triangle  action}  on weakly increasing binary trees.

For any tree $B\in\mathcal{B}_M$, let $[B]=\{g(B): g\in\mathbb{Z}_2^p\}$ be the orbit of $B$ under the triangle  action. The triangle  action divides the set $\mathcal{B}_M$ into disjoint orbits. Let us introduce
\begin{align*}
\Gamma_{M}&:=\{B\in\mathcal{B}_M: \eact(B)=0\},  \\
 \Gamma_{M,i,j}&:=\{B\in\Gamma_M: \eler(B)=i\text{ and } \mathsf{act}(B)=j\}.
\end{align*}
Because of Lemma~\ref{LEM:bi-even->odd}, each orbit $[B]$ contains a unique tree from $\Gamma_{M}$ that we denote $\tilde B$. Moreover, for each $\tilde B\in\Gamma_M$, we have  
\begin{align*}
&\quad\sum_{B\in[\tilde B]}x^{\eler(B)}y^{\oler(B)}z^{\ord(B)}\\
&=\sum_{B\in[\tilde B]}x^{\eler(B)}y^{\dme(B)+\ndoler(B)}z^{\dmo(B)+\ndord(B)}\qquad\text{(by~\eqref{rel:ndoler})}\\
&=\sum_{B\in[\tilde B]}x^{\eler(B)}y^{\dme(B)}z^{\dmo(B)}(yz)^{\ndord(B)}\qquad\text{(by Lemma~\ref{lem:main})}\\
&=x^{\eler(\tilde B)}(yz)^{\ndord(\tilde B)}(y^2+z^2)^{\act(\tilde B)}.\qquad\text{(by~\eqref{dem:eact} and Lemmas~\ref{LEM:bi-eler-action} and~\ref{LEM:dym})}
\end{align*}
Summing over all orbits gives
\begin{align*}
\sum_{B\in\B_M}x^{\eler(B)}y^{\oler(B)}z^{\ord(B)}&=\sum_{\tilde B\in\Gamma_M}\sum_{B\in[\tilde B]}x^{\eler(B)}y^{\oler(B)}z^{\ord(B)}\\
&=\sum_{\tilde B\in\Gamma_M}x^{\eler(\tilde B)}(yz)^{\ndord(\tilde B)}(y^2+z^2)^{\act(\tilde B)}\\
&=\sum_{i}x^{i}\sum_{j=0}^{\frac{p+1-i}{2}}|\Gamma_{M,i,j}|(yz)^{\frac{p+1-i}{2}-j}(y^2+z^2)^{j},
\end{align*}
where the last equality follows from $p+1=2(\ndord(\tilde B)+\act(\tilde B))+\eler(\tilde B)$, a consequence of $p+1=\oler(\tilde B)+\ord(\tilde B)+\eler(\tilde B)$, relationships in~\eqref{rel:ndoler} and  Lemma~\ref{lem:main}. The above expansion is exactly~\eqref{wibt:partial}, which completes the proof of the theorem.
\end{proof}

 \section{Final remarks}
 In this paper, we find the symmetry~\eqref{Weak:thm} on weakly increasing trees on a multiset and provide three proof of it: an involution proof, a generating function proof and a group action proof, each of which has its own merit. At this point, we would like to pose a challenging conjecture regarding the reduced Schett polynomials $\hat{S}_M(x,y,z)$ defined in~\eqref{def:redsche}.
 
 \begin{conjecture} Let 
 $$
 \hat{S}_{M,i}(t):=\sum_{T\in\mathcal{T}_M\atop\lfloor\ee(T)/2\rfloor=i}t^{\lfloor\frac{\oe(T)}{2}\rfloor}.
 $$
 Then the polynomial $\hat{S}_{M,i}(t)$ has only real roots for any  $M$ and $i$. 
 \end{conjecture}
 
 Br\"and\'en~\cite[Remark~7.3.1]{Br2} observed that if a polynomial $h(t)\in\mathbb{R}[t]$ is  palindromic  and has only real roots, then $h(t)$ is $\gamma$-positive. Since Theorem~\ref{thm:action} implies the $\gamma$-positivity of $\hat{S}_{M,i}(t)$, it provides partial evidence for the above real-rootedness conjecture. Moreover, we have verified the conjecture for plane trees and increasing trees of nodes less than eleven.

\section*{Acknowledgement}

The first author was supported by the National Science Foundation of China grant 11871247 and the project of Qilu Young Scholars of Shandong University. 

\end{document}